\pdfoutput=1
\documentclass[a4paper, 11pt, reqno]{amsart}

\usepackage{amssymb}
\usepackage{amsmath}
\usepackage{amsthm}
\usepackage[numbers]{natbib}
\usepackage{mathtools}
\usepackage{enumitem}
\usepackage{calc}
\usepackage{setspace}
\usepackage{graphicx}
\usepackage{tikz-cd}
\usepackage{mathscinet}
\usepackage{hyperref}
\newcommand{\noop}[1]{}
\usepackage[top=3 cm, bottom=3 cm, left=3 cm, right=3cm]{geometry}

\tikzset{
  symbol/.style={
    draw=none,
    every to/.append style={
      edge node={node [sloped, allow upside down, auto=false]{$#1$}}}
  }
}

\newcounter{dummy}
\numberwithin{dummy}{section}
\newtheorem{thm}[dummy]{Theorem}
\newtheorem{defn}[dummy]{Definition}

\newtheorem{lem}[dummy]{Lemma}
\newtheorem{prop}[dummy]{Proposition}
\newtheorem{cor}[dummy]{Corollary}

\theoremstyle{definition}
\newtheorem{rmk}[dummy]{Remark}
\numberwithin{equation}{section}

\DeclareMathOperator{\id}{id}

\DeclareMathOperator{\ad}{ad}

\DeclareMathOperator{\ws}{ws}
\DeclareMathOperator{\red}{red}

\DeclareMathOperator{\GL}{\mathrm{GL}}

\DeclareMathOperator{\Spec}{\mathrm{Spec}}

\DeclareMathOperator{\Hom}{\mathrm{Hom}}

\DeclareMathOperator{\ZP}{\mathrm{ZP}}
\DeclareMathOperator{\WML}{\mathrm{WML}}
\DeclareMathOperator{\SML}{\mathrm{SML}}

\makeatletter
\renewcommand{\@biblabel}[1]{[#1]\hfill}
\makeatother

\begin{document}
\title{Hecke orbits and the Mordell-Lang conjecture in distinguished categories}

\author[F. Barroero]{Fabrizio Barroero}
\address[F. Barroero]{Universit\`a degli studi Roma Tre, Dipartimento di Matematica e Fisica, Largo San Murialdo 1, 00146 Roma, Italy}
\email{fbarroero@gmail.com}

\author[G. A. Dill]{Gabriel A. Dill}
\address[G. A. Dill]{Institut de Math\'ematiques, Universit\'e de Neuch\^atel, Rue Emile-Argand 11, 2000 Neuch\^atel, Switzerland}
\email{gabriel.dill@unine.ch}

\date{\today}

\begin{abstract}
Inspired by recent work of Aslanyan and Daw, we introduce the notion of $\Sigma$-orbits in the general framework of distinguished categories. In the setting of connected Shimura varieties, this concept contains many instances of (generalized) Hecke orbits from the literature. In the setting of semiabelian varieties, a $\Sigma$-orbit is a subgroup of finite rank. We show that our $\Sigma$-orbits have useful functorial properties and we use them to formulate two general statements of Mordell-Lang type (one of them implying the other one). We prove an analogue of a recent theorem of Aslanyan and Daw in this general setting, which we apply to deduce an unconditional result about unlikely intersections in a fibered power of the Legendre family. In an appendix, we prove an unconditional Zilber-Pink result for subvarieties of $\mathcal{A}_g$ that cannot be defined over the algebraic numbers.
\end{abstract}

\maketitle

\tableofcontents

\section{Introduction}\label{sec:intro}

In \cite{BD2} we have introduced distinguished categories with the aim of providing a general formal framework in which certain statements and facts about unlikely intersections can be formulated and proved at the same time in several different settings, e.g., for semiabelian varieties and connected mixed Shimura varieties. In that paper we focused on the Zilber-Pink conjecture in its different formulations and proved several facts about it. See the book \cite{Zannier} for a survey on unlikely intersection problems and the Zilber-Pink conjecture.

In this paper, which is inspired by \cite{AD22}, we formulate analogues of the Mordell-Lang conjecture, which we call the strong and the weak Mordell-Lang statement respectively, in a general distinguished category. For this, we introduce the notion of $\Sigma$-orbits. In the setting of semiabelian varieties, $\Sigma$-orbits are subgroups of finite rank and both the strong and the weak Mordell-Lang statement follow from the classical Mordell-Lang conjecture, which is a theorem due to McQuillan \cite{McQ}. In the setting of connected pure Shimura varieties, we will see that our notion of a $\Sigma$-orbit contains various notions of (generalized) Hecke orbits from the literature. In particular, we will see in Subsection \ref{sub:csv1} that, for connected pure Shimura varieties, our notion of a $\Sigma$-orbit is equivalent to the notion of hybrid Hecke orbit introduced in \cite{Richard_Yafaev_24a}. Our weak Mordell-Lang statement for connected pure Shimura varieties implies the Mordell-Lang conjecture for Shimura varieties formulated in \cite{AD22}; for connected pure Shimura varieties of abelian type, our weak Mordell-Lang statement has in the meantime been proved in \cite{Richard_Yafaev_24b}.

The main result of this article, which is Theorem \ref{thm:AD1}, is an analogue of Theorem 4.3 in \cite{AD22} (see also Remark \ref{rmk:sigmaatypical}). Combined with results from the literature, Theorem \ref{thm:AD1} implies Theorem 1.8 of \cite{Aslanyan} and unconditional results for subvarieties of connected pure Shimura varieties of abelian type (Corollary \ref{cor:shimura}) as well as for subvarieties of fibered powers of the Legendre elliptic scheme (Corollary \ref{cor:maincor}).

Thanks to a discussion at the third JNT biennial conference, we realized that the methods employed in \cite{BD2} and the present article also lead to a resolution of some cases of the Zilber-Pink conjecture for connected pure Shimura varieties with simple adjoint group. We include the proof of the relevant statement about distinguished categories in Appendix \ref{appendix}. The following theorem is a special case of Theorem \ref{thm:transcendentalzilberpinkshimura} (see Remark \ref{rmk:transcendentalzilberpinka_g}). We use $\mathcal{A}_g$ to denote the coarse moduli space of principally polarized abelian varieties of dimension $g$ over $\bar{\mathbb{Q}}$ ($g \in \mathbb{N}$).

\begin{thm}\label{thm:transcendentalzilberpinka_g}
Let $g \in \mathbb{N}$ and let $K$ be an algebraically closed field of characteristic $0$. If $V \subset (\mathcal{A}_g)_K$ is an irreducible closed subvariety that is not the base change of a closed subvariety of $\mathcal{A}_g$ (in particular, $K \neq \bar{\mathbb{Q}}$) and is not contained in a proper special subvariety of $(\mathcal{A}_g)_K$, we have that the union of the intersections of $V$ with all special subvarieties of $(\mathcal{A}_g)_K$ of codimension $> \dim V$ is not Zariski dense in $V$.
\end{thm}

If $g = 1$, Theorem \ref{thm:transcendentalzilberpinka_g} holds trivially while, if $g = 2$ and $\dim V = 1$, we have proved it in \cite{BD2} (if $g = 2$ and $\dim V = 2$, Theorem \ref{thm:transcendentalzilberpinka_g} again holds trivially since special points are algebraic).

We now give an overview of the content of this article.

After recalling some preliminaries in Section \ref{sec:preliminaries}, we introduce in Section \ref{sec:sigmaorbits} the notion of $\Sigma$-orbits and prove several basic facts about them.
In Section \ref{sec:exi}, we first describe $\Sigma$-orbits in the semiabelian case and see that they are finite rank subgroups. We then consider the setting of connected pure Shimura varieties and relate the concept of $\Sigma$-orbits with some definitions of (generalized) Hecke orbits appearing in the literature.

Section \ref{sec:sigmaspecial} is devoted to $\Sigma$-special subvarieties, which are weakly special subvarieties coming from a point of a $\Sigma$-orbit $\Sigma$: after their definition, we prove a very useful characterization (Lemma \ref{lem:sigmaspecialchar}) as well as several important facts. We moreover define $\Sigma_Z$-point-special subvarieties of a distinguished variety $Z$ (weakly special subvarieties containing a point of $\Sigma$) and show that they are $\Sigma$-special in most examples.
In Section \ref{sec:exii}, we consider the converse problem: we show that $\Sigma$-special subvarieties are $\Sigma_Z$-point-special for almost-split semiabelian varieties and connected pure Shimura varieties, but not in the general setting of semiabelian varieties. 

Section \ref{sec:ml} starts with the statements of the two above mentioned formulations of the Mordell-Lang statement, one of them implying the other one. We then show that they follow from the Zilber-Pink statement (as defined in Definition 10.1 of \cite{BD2}) and state and prove our main result, Theorem  \ref{thm:AD1}.

Finally, in Section \ref{sec:applications}, we show that the usual Mordell-Lang conjecture for semiabelian varieties (McQuillan's Theorem \cite{McQ}) is equivalent to our formulations, we recover Theorem 1.8 of \cite{Aslanyan}, and we deduce an unconditional result (Corollary \ref{cor:shimura}) for connected pure Shimura varieties of abelian type from our Theorem \ref{thm:AD1} together with a recent result of Richard and Yafaev \cite{Richard_Yafaev_24b}. We also deduce an unconditional result (Corollary \ref{cor:maincor}) for subvarieties of fibered powers of the Legendre elliptic scheme from our Theorem \ref{thm:AD1} and results of Gao as well as of the second-named author \cite{G17,Dill20,Dill21}.

For results and conjectures related to our weak Mordell-Lang statement for connected pure Shimura varieties, see \cite{OrrThesis,Orr15,RY17,Aslanyan,qiu22,RYa,Richard_Yafaev,AD22,Richard_Yafaev_24a,Richard_Yafaev_24b}.

In this article, all fields are of characteristic $0$. Varieties and curves are irreducible, subvarieties are closed. While (mixed) Shimura varieties as usually defined might not be irreducible, we consider here only connected (mixed) Shimura varieties, which are irreducible. We denote by $K$ an algebraically closed field of characteristic $0$. Varieties will always be varieties over $K$ if their ground field is not mentioned explicitly. 

\section{Preliminaries}\label{sec:preliminaries}

In this section, we recall some definitions from \cite{BD2} and prove some basic facts.

Consider a category $\mathfrak{C}$ with objects $\mathfrak{V}$ and morphisms $\mathfrak{M}$ together with a covariant functor $\mathcal{F}$ from $\mathfrak{C}$ to the category of varieties over $K$. Typically, this functor will correspond to forgetting some additional information.

We formulate the following axioms that this category may or may not satisfy:

\begin{enumerate}[label={(A\arabic*)}]
	\item \label{ax:1}\emph{Direct Products} - If $X,Y \in \mathfrak{V}$, then there exists $Z \in \mathfrak{V}$ and there exist morphisms $\pi_X:Z \to X$, $\pi_Y:Z \to Y$ in $\mathfrak{M}$ such that the morphism $\mathcal{F}(Z) \to \mathcal{F}(X) \times_K \mathcal{F}(Y)$ induced by the morphisms $\mathcal{F}(\pi_X)$ and $\mathcal{F}(\pi_Y)$ is an isomorphism. We identify $\mathcal{F}(Z)$ with $\mathcal{F}(X) \times_K \mathcal{F}(Y)$. Furthermore, if $\phi: W \to X$ and $\psi: W \to Y$ are morphisms in $\mathfrak{M}$, then there exists a unique morphism $\chi: W \to Z$ in $\mathfrak{M}$ such that $\pi_X\circ \chi=\phi$ and $\pi_Y\circ \chi=\psi$. We write $X \times Y$ or $X \times_K Y$ for $Z$ and $(\phi,\psi)$ for $\chi$. 
	\item\label{ax:2} \emph{Fibered Products} - If $\phi: X \to Z$ and $\psi: Y \to Z$ are morphisms in $\mathfrak{M}$, then there exists $n \in \mathbb{Z}$, $n \geq 0$, and there exist $X_1,\hdots,X_n \in \mathfrak{V}$ and morphisms $\phi_i: X_i \to X \times Y$ in $\mathfrak{M}$ ($i=1,\hdots,n$) such that $\bigcup_{i=1}^{n}{\mathcal{F}(\phi_i)(\mathcal{F}(X_i))} = (\mathcal{F}(X) \times_{\mathcal{F}(Z)} \mathcal{F}(Y))_{\red} \subset \mathcal{F}(X) \times_K \mathcal{F}(Y)$.
	\item\label{ax:3} \emph{Final Object} - The category has a final object that is mapped to $\Spec K$ by $\mathcal{F}$.
	\item \label{ax:4} \emph{Fiber Dimension} - If $\phi: X \to Y$ is a morphism in $\mathfrak{M}$, then $\mathcal{F}(\phi)(\mathcal{F}(X))$ is closed in $\mathcal{F}(Y)$ and there exist morphisms $\phi_1: W \to X$, $\phi_2: W \to Z$, and $\phi_3: Z \to Y$ in $\mathfrak{M}$ such that $\phi \circ \phi_1 = \phi_3 \circ \phi_2$, $\mathcal{F}(\phi_1)$ is finite and surjective, $\mathcal{F}(\phi_3)$ has finite fibers, $\mathcal{F}(\phi_2)$ is surjective, $\mathcal{F}(\phi_2)^{-1}(z)$ is irreducible for all $z \in \mathcal{F}(Z)$, and
	$\dim_w \mathcal{F}(\phi_2)^{-1}(\mathcal{F}(\phi_2)(w))$ is a constant function for $w \in \mathcal{F}(W)$.
	\end{enumerate}

\begin{defn}
	If $\mathfrak{C}$ satisfies \ref{ax:1} to \ref{ax:4}, we call it a \emph{distinguished category}. The elements of $\mathfrak{V}$ (which we will often identify with their images under $\mathcal{F}$) will then be called \emph{distinguished varieties} while the elements of $\mathfrak{M}$ (which we will often similarly identify) will be called \emph{distinguished morphisms}.
\end{defn}

Note that \ref{ax:4} implies that $\dim_x \phi^{-1}(\phi(x))$ is a constant function for $x\in X$ (see Section 2 of \cite{BD2}). If $L$ is an algebraically closed field that contains $K$ and $\mathfrak{C}$ is a distinguished category over $K$, then we obtain a distinguished category $\mathfrak{C}_L$ over $L$ by composing $\mathcal{F}$ with the base change functor from $K$ to $L$.

\begin{defn}
	Let $Z$ be a distinguished variety. A \emph{special} subvariety of $Z$ is the image of a distinguished morphism.
	A point $z \in Z(K)$ is called \emph{special} if the singleton $\{z\}$ is a special subvariety of $Z$.
\end{defn}

\begin{lem}\label{lem:finfibspecial}
	Any special subvariety is the image of a distinguished morphism with finite fibers. Any special point is the image of a distinguished morphism whose domain is a point.
\end{lem}

\begin{proof}
	The first part of the lemma is a consequence of \ref{ax:4}. Indeed, in \ref{ax:4}, since $\phi_1$ and $\phi_2$ are surjective, we have $ \phi(X) =  \phi_3 (Z)$.
	
	The second part of the lemma follows from the fact that a variety whose underlying topological space is finite must be a point.
\end{proof}

\begin{lem}\label{lem:finfiba2}
	The distinguished morphisms $\phi_1,\dots ,\phi_n$ in \ref{ax:2} may be assumed to have finite fibers.
\end{lem}

\begin{proof}
	The $\phi_i(X_i) $ in \ref{ax:2} are all special subvarieties of $X \times Y$, so the lemma follows from Lemma \ref{lem:finfibspecial}.
\end{proof}

\begin{defn}
	Let $Z$ be a distinguished variety. A subvariety $W$ of $Z$ is called \emph{weakly special} if there exist distinguished morphisms $\psi: Y \to X$, $\phi: Y \to Z$ and a point $x \in X(K)$ such that $W$ is an irreducible component of $\phi(\psi^{-1}(x))$.
\end{defn}

We showed in \cite{BD2} that irreducible components of intersections of (weakly) special subvarieties are (weakly) special. We can therefore make the following definition.

\begin{defn}
	For a subvariety $V$ of a distinguished variety $X$, let $\langle V \rangle$ denote the smallest special subvariety of $X$ that contains $V$ and let $\langle V \rangle_{\ws}$ denote the smallest weakly special subvariety of $X$ that contains $V$. We call $\langle V \rangle$ the \emph{special closure} of $V$ and $\langle V \rangle_{\ws}$ the \emph{weakly special closure} of $V$. Furthermore, we call $\delta(V) = \dim \langle V \rangle - \dim V$ the \emph{defect} of $V$ and $\delta_{\ws}(V) = \dim \langle V \rangle_{\ws} - \dim V$ the \emph{weak defect} of $V$.
\end{defn}

\begin{defn}
Let $V$ be a subvariety of a distinguished variety $X$. A subvariety $W \subset V$ is called \emph{optimal for $V$ in $X$} if $\delta(U) > \delta(W)$ for every subvariety $U$ such that $W \subsetneq U \subset V$. It is called \emph{weakly optimal for $V$ in $X$} if it satisfies the same property with $\delta_{\ws}$ in place of $\delta$.
\end{defn}

We can now formulate a fifth axiom that a distinguished category may or may not satisfy.

\begin{enumerate}[label={(A\arabic*)}]
	\setcounter{enumi}{4}
	\item \label{ax:5} \emph{Weak Finiteness} - If $X$ is a distinguished variety and $V \subset X$ is a subvariety, then there exists a finite set of pairs $(\phi,\psi)$ of distinguished morphisms $\phi: Y_\phi \to X$ and $\psi: Y_\phi \to Z_\psi$ such that for every subvariety $W \subset V$ that is weakly optimal for $V$ in $X$ there exists a pair $(\phi,\psi)$ in this set and $z \in Z_\psi(K)$ such that $\phi$ has finite fibers and $\langle W \rangle_{\ws}$ is an irreducible component of $\phi(\psi^{-1}(z))$.
\end{enumerate}

In \cite{BD2}, several examples of distinguished categories are given. Most of them also satisfy \ref{ax:5}. In particular, semiabelian varieties, connected pure Shimura varieties, and connected mixed Shimura varieties of Kuga type all form distinguished categories that satisfy \ref{ax:5} and to which we will apply the results of this article. In Sections \ref{sec:sigmaorbits}, \ref{sec:sigmaspecial}, and \ref{sec:ml}, we always work in a fixed ambient distinguished category unless stated otherwise.

\section{$\Sigma$-orbits}\label{sec:sigmaorbits}

\begin{defn}
	For $X, Z\in \mathfrak{V}$ and $x\in X(K)$ we define the \emph{$\Sigma$-orbit $\Sigma_Z(x)$ of $x$ in $Z$} to be the set of all $\phi(y)$ where $Y\in \mathfrak{V}$, $\phi: Y \to Z$ is a distinguished morphism, and $y\in Y(K)$ such that there exists a distinguished morphism $\psi: Y \to X$ with finite fibers with $\psi(y) = x$.
	
	We call the collection of all $\Sigma_Z(x)$, for varying $Z\in \mathfrak{V}$, the \emph{$\Sigma$-orbit $\Sigma(x)$ of $x$}.
\end{defn}

Obviously, if $\phi: Z \to Z'$ is a distinguished morphism and $z \in \Sigma_Z(x)$, then $z'=\phi(z) \in \Sigma_{Z'}(x)$. In the following, we see that our $\Sigma$-orbits enjoy many other nice properties and in the next section we are going to see some examples.

\begin{prop}\label{prop:heckeunderfinitecover} Let $X\in \mathfrak{V}$ and $x\in X(K)$. 
	 For a distinguished morphism $\chi:Z'\to Z$ with finite fibers we have
		 $\chi^{-1}(\Sigma_Z(x))=\Sigma_{Z'}(x)$.

\end{prop}

\begin{proof}
	The inclusion $\supset$ is obvious. Let $z'\in \chi^{-1}(\Sigma_Z(x))$ and set $z=
	\chi(z')$. By definition there are distinguished morphisms $\phi: Y \to Z$ and $\psi: Y \to X$, where the latter has finite fibers, and there is $y\in Y(K)$ such that $\psi(y)=x$ and $\phi(y)=z$. We consider the fibered product $(X\times Z')\times_{X\times Z}Y$ given by the distinguished morphisms $\id_X\times \chi:X\times Z'\to X\times Z$ and $(\psi,\phi):Y\to X\times Z$.
	
	We apply \ref{ax:2} and Lemma \ref{lem:finfiba2} to find distinguished morphisms $\phi_i: X_i \to X\times Z'\times Y$, $i=1, \dots ,n$, with finite fibers such that $\bigcup_{i=1}^{n}{\phi_i(X_i)} = ((X\times Z')\times_{X\times Z}Y)_{\red}$. We may assume that $(x,z',y)\in \phi_1(X_1(K))$ and pick $x_1\in X_1(K)$ such that $\phi_1(x_1)=(x,z',y)$.  After composing with the appropriate projection, we get distinguished morphisms $\phi': X_1 \to Z'$ and $\psi': X_1 \to X$, where $\psi'$ has finite fibers, such that $\phi'(x_1)=z'$ and $\psi'(x_1)=x$. This implies that $z'\in \Sigma_{Z'}(x)$.
\end{proof}

\begin{prop}\label{prop:heckeunderfinitecover2} Let $\chi: X' \to X$ be a distinguished morphism with finite fibers, let $x' \in X'(K)$, and set $x = \chi(x') \in X(K)$. 
	 For every distinguished variety $Z$ we have
		 $\Sigma_Z(x) =\Sigma_Z(x')$. Hence $\Sigma(x)=\Sigma(x')$.

\end{prop}

\begin{proof}
	The inclusion $\supset$ is obvious, so suppose that $z \in \Sigma_Z(x)$. By definition there are distinguished morphisms $\phi: Y \to Z$ and $\psi: Y \to X$, where the latter has finite fibers, and there is $y\in Y(K)$ such that $\psi(y)=x$ and $\phi(y)=z$. We consider the fibered product $X' \times_{X} Y$.
	
	We apply \ref{ax:2} and Lemma \ref{lem:finfiba2} to find distinguished morphisms $\phi_i: X_i \to X' \times Y$, $i=1, \dots ,n$, with finite fibers such that $\bigcup_{i=1}^{n}{\phi_i(X_i)} = (X' \times_{X} Y)_{\red}$. We may assume that $(x',y)\in \phi_1(X_1(K))$ and pick $x_1\in X_1(K)$ such that $\phi_1(x_1)=(x',y)$.  After composing with the appropriate projection, we get distinguished morphisms $\psi': X_1 \to X'$ and $\phi': X_1 \to Y$, both having finite fibers, such that $\psi'(x_1)=x'$ and $\phi'(x_1)=y$. It follows that $(\phi \circ \phi')(x_1) = z$. We deduce that $z \in \Sigma_{Z}(x')$.
\end{proof}

\begin{prop}\label{prop:heckeunderproduct}
	Let $X\in \mathfrak{V}$ and $x\in X(K)$. For any two distinguished varieties $Y$ and $Z$, we have
\[ \Sigma_Y(x) \times \Sigma_Z(x) = \Sigma_{Y \times Z}(x).\]
\end{prop}

\begin{proof}
The inclusion $\supset$ is obvious. Let $y \in \Sigma_Y(x)$ and $z \in \Sigma_Z(x)$. By definition there are distinguished morphisms $\phi_Y: Y' \to Y$, $\phi_Z: Z' \to Z$, $\psi_Y: Y' \to X$, and $\psi_Z: Z' \to X$, where $\psi_Y$ and $\psi_Z$ have finite fibers. Furthermore, there are $y'\in Y'(K)$ and $z' \in Z'(K)$ such that $\psi_Y(y')= \psi_Z(z') = x$, $\phi_Y(y')=y$, and $\phi_Z(z') = z$.

We consider the fibered product $(Y' \times Z') \times_{X \times X} X$, where the morphism $X \to X \times X$ is the diagonal embedding. We apply \ref{ax:2} and Lemma \ref{lem:finfiba2} to find distinguished morphisms $\chi_i: X_i \to Y\times Z'\times X$, $i=1, \dots ,n$, with finite fibers such that $\bigcup_{i=1}^{n}{\chi_i(X_i)} = ((Y' \times Z') \times_{X \times X} X)_{\mathrm{red}}$. We may assume that $(y',z',x)\in \chi_1(X_1(K))$ and pick $x_1\in X_1(K)$ such that $\chi_1(x_1)=(y',z',x)$. After composing with the appropriate projection, we get distinguished morphisms $\phi_{X_1}: X_1 \to Y' \times Z'$ and $\psi_{X_1}: X_1 \to X$, where $\psi_{X_1}$ has finite fibers, such that $\phi_{X_1}(x_1)=(y',z')$ and $\psi_{X_1}(x_1)=x$. Hence $(y',z') \in \Sigma_{Y' \times Z'}(x)$ and therefore $(y,z) \in \Sigma_{Y \times Z}(x)$.
\end{proof}

\begin{prop}\label{prop:heckeandspecialpoint}
	Let $X\in \mathfrak{V}$ and $x\in X(K)$. For any distinguished variety $Z$ and any special point $z \in Z(K)$, we have that $z \in \Sigma_Z(x)$.
\end{prop}

\begin{proof}
By Lemma \ref{lem:finfibspecial}, there exists a zero-dimensional distinguished variety $Y = \{y\}$ and a distinguished morphism $\phi: Y \to Z$ such that $z = \phi(y)$. The projection $\pi: Y \times X \to X$ has finite fibers and $(y,x) \in (Y \times X)(K)$ satisfies that $\pi(y,x) = x$ and $\phi(y) = z$. Hence $z \in \Sigma_Z(x)$.
\end{proof}

\begin{rmk}
	
	As images of special points under distinguished morphisms are special points and preimages of special points under distinguished morphisms with finite fibers consist of special points (see Lemma 6.7 of \cite{BD2}), we note that the $\Sigma$-orbit of a special point in a distinguished variety $Z$ is the set of special points of $Z$. This implies that, for a distinguished variety $X$, as soon as $x$ is not a special point of $X$, it does not lie in the $\Sigma$-orbit of any special point $z\in Z(K)$ even if $z\in \Sigma_Z(x)$ and therefore, in general, for a distinguished variety $Y$, $\Sigma$-orbits in $Y$ do not give a partition of $Y(K)$.
\end{rmk}

\begin{lem}\label{lem:heckeorbitintersection2}
	Let $X$ and $Z$ be distinguished varieties. Let $x \in X(K)$ and $z \in Z(K)$. The following are equivalent:
	\begin{enumerate}
		\item $z \in \Sigma_Z(x)$, and
		\item $\{(z,x)\}$ is an irreducible component of the intersection of $Z \times \{x\}$ with a special subvariety of $Z \times X$.
	\end{enumerate}
\end{lem}

\begin{proof}
	By definition there are distinguished morphisms $\phi: Y \to Z$ and $\psi: Y \to X$, where the latter has finite fibers, and there is $y\in Y(K)$ such that $\psi(y)=x$ and $\phi(y)=z$. By \ref{ax:1} the morphism $\chi=(\phi,\psi):Y\to Z\times X$ is distinguished and therefore $\chi(Y)$ is special. Since $\psi$ has finite fibers $Z\times \{x\}\cap \chi(Y)$ is a finite set and thus $\{(z,x)=\chi(y)\}$ is an irreducible component of it.
	
	Suppose that (2) holds. By Lemma \ref{lem:finfibspecial}, there exists a distinguished morphism $\chi: Y \to Z \times X$ with finite fibers such that $\{(z,x)\}$ is an irreducible component of $\chi(Y) \cap (Z \times \{x\})$.
	
	It follows from \cite[Lemma 7.4]{BD2} that the restriction of the projection $Z \times X \to X$ to $\chi(Y)$ has finite fibers. Since $\chi$ has finite fibers, also the induced morphism $\psi: Y \to X$ has finite fibers. It follows that $(z,x) \in \chi(\psi^{-1}(x)) \subset \Sigma_{Z \times X}(x)$ and hence $z \in \Sigma_Z(x)$.
\end{proof}

\section{Examples I}\label{sec:exi}

\subsection{Semiabelian varieties}  \label{subsec:semiab1}
We consider the distinguished category of semiabelian varieties over $K$. We recall that distinguished morphisms are homomorphisms of algebraic groups composed with translations by torsion points and that weakly special subvarieties of a semiabelian variety $G$ are cosets of semiabelian subvarieties of $G$.

Given a subgroup $\Gamma$ of $G(K)$ we indicate by $\Gamma^{\mathrm{div}}$ the \emph{division group} of $\Gamma$, meaning the subgroup of $G(K)$ of points that have a multiple in $\Gamma$.

Let $G, G_0$ be semiabelian varieties over $K$ and $x\in G_0(K)$. We want to describe the $\Sigma$-orbit $\Sigma_G(x)$ of $x$ in $G$.

We start by noticing that there is a semiabelian variety $G_0'$ and an injective distinguished morphism $\phi_0: G_0' \to G_0$ such that $\phi_0(G_0')$ is the smallest special subvariety of $G_0$ that contains $x$. Because of Proposition \ref{prop:heckeunderfinitecover2}, we can replace $G_0$ by $G_0'$ and $x$ by its unique preimage under $\phi_0$ without changing $\Sigma(x)$. Having done so, we assume from now on that $x$ does not belong to any proper special subvariety of $G_0$.

The $\mathbb{Z}$-module $\Hom(G_0,G)$ is well known (see Lemma 1 in \cite{Lars} for instance) to be  finitely generated. We let $\psi_1,\dots ,\psi_n$ be generators of $\Hom(G_0,G)$. We claim that $\Sigma_G (x)= \langle \psi_1(x),\dots, \psi_n(x) \rangle^{\mathrm{div}}$.

We first show ``$\subset$". Let $g\in \Sigma_G(x)$. There are distinguished morphisms $\psi:Y\to G$ and $\phi:Y\to G_0$ such that $\phi $ has finite fibers and there is a $y\in Y(K)$ with $\psi(y)=g$ and $\phi(y)=x$. First, note that $\phi$ must be surjective since $\phi(Y)$ is a special subvariety of $G_0$ that contains $x$.
By Proposition \ref{prop:heckeunderfinitecover2}, for any translate $x'$ of $x$ by a torsion point of $G_0$, we have $\Sigma(x)=\Sigma(x')$. Since we also have that $\langle \psi_1(x),\dots, \psi_n(x) \rangle^{\mathrm{div}} = \langle \psi_1(x'),\dots, \psi_n(x') \rangle^{\mathrm{div}}$ for any such $x'$, we may and will replace $x$ by some such $x'$ and assume without loss of generality that $\phi$ is a homomorphism of algebraic groups. As a surjective homomorphism with finite fibers, $\phi$ is an isogeny. By Lemma 7.3/5 on p. 180 of \cite{Neronmodels}, we may replace $Y$ by $G_0$ and $\phi$ by the multiplication by some non-zero integer. Then, $g=t+\sum_{i=1}^n a_i \psi_i(y)$ for some torsion point $t$ and integers $a_1, \dots ,a_n$ and $Ny=x$ for some non-zero integer $N$. This clearly implies that $g\in \langle \psi_1(x),\dots, \psi_n(x) \rangle^{\mathrm{div}}$.

For the other inclusion, it is enough to show that, for all $g_1,g_2\in G(K)$ and $N\in \mathbb{Z}\setminus \{0\}$ we have $$g_1 \in \Sigma_G(x)\iff Ng_1 \in \Sigma_G(x) \ \ \mbox{ and } \ \ g_1, g_2 \in \Sigma_G(x) \implies g_1+g_2\in \Sigma_G(x).$$

The non-obvious direction of the equivalence follows from Proposition \ref{prop:heckeunderfinitecover}.  The implication can easily be deduced using Proposition \ref{prop:heckeunderproduct}.

We have just shown that $\Sigma_G(x)$ is a finite rank subgroup of $G(K)$. On the other hand, for a finite rank subgroup $\Gamma$ of $G(K)$, we would like to find a $\Sigma$-orbit that contains it.
Let $x_1,\dots , x_n\in G(K)$ be such that $\Gamma\subset \langle x_1,\dots, x_n \rangle^{\mathrm{div}}$ and set $x=(x_1,\dots, x_n) \in G^n(K)$.
It is easy to see that for any $\gamma \in \Gamma$ {the singleton $\{(\gamma,x)\}$} is {an irreducible} component of the intersection of $G\times \{x\}$ with a special subvariety of $G\times G^n$. By Lemma \ref{lem:heckeorbitintersection2}, we deduce that $\gamma\in \Sigma_G(x)$.

\subsection{Connected pure Shimura varieties} \label{sub:csv1}
At first glance, our definition of $\Sigma$-orbits in the case of connected pure/mixed Shimura varieties resembles the definition of a generalized Hecke orbit which can be found in Proposition-Definition 3.3 in \cite{Pink}. In this subsection, we compare these two notions and some other notions of generalized Hecke orbits that appear throughout the literature.

Several slightly different definitions of (connected) pure/mixed Shimura data are used in the literature. We will always use the definition from \cite{BD2} and translate all definitions of generalized Hecke orbits to this setting.

In the following, we freely use terminology from Section 3 of \cite{BD2}. Let $(P,X^+,\Gamma)$ be an object of the distinguished category $\mathfrak{C}_{\mathrm{pSv}}$ of connected pure Shimura varieties with associated functor $\mathcal{F}_{\mathrm{pSv}}$. Set $S = \mathcal{F}_{\mathrm{pSv}}(P,X^+,\Gamma)$ and let $x \in S(\mathbb{C})$. There is a canonical identification of $S(\mathbb{C})$ with $\Gamma \backslash X^+$. Let $\widetilde{x}$ be a lift of $x$ to $X^+ \subset \mathrm{Hom}(\mathbb{S}_{\mathbb{C}},P_{\mathbb{C}})$, where $\mathbb{S}$ denotes the Deligne torus, let $M$ denote the smallest algebraic subgroup of $P$ (defined over $\mathbb{Q}$) such that the image of $\widetilde{x}$ is contained in $M_{\mathbb{C}}$, and let $\widetilde{x}_0: \mathbb{S}_{\mathbb{C}} \to M_{\mathbb{C}}$ denote the induced homomorphism. Since $M$ is automatically connected, it follows from Lemma 4.3 in \cite{BD2} that there is a Shimura subdatum $(M,X_M^+)$ of $(P,X^+)$ such that $\widetilde{x}_0 \in X_M^+$. Thanks to Lemma 2.2 in \cite{OrrThesis} and Proposition 2.2 in \cite{RohlfsSchwermer}, we find an object $(M,X_M^+,\Gamma_M)$ of $\mathfrak{C}_{\mathrm{pSv}}$ and a morphism $(M,X_M^+,\Gamma_M) \to (P,X^+,\Gamma)$ of $\mathfrak{C}_{\mathrm{pSv}}$ extending the Shimura morphism $(M,X_M^+) \to (P,X^+)$.

Let $y \in S(\mathbb{C})$ and let $\widetilde{y}$ be a lift of $y$ to $X^+$. In the literature, we find the following properties that $\widetilde{y}$ should satisfy for $y$ to belong to the generalized Hecke orbit of $x$ (translated to the setting of \cite{BD2}):
\begin{itemize}
	\item for some faithful representation $\rho: P \to \GL(V)$ of $P$, the representations $\rho \circ \widetilde{x}$ and $\rho \circ \widetilde{y}$ are conjugate to each other by an element of $\GL(V)(\mathbb{Q})$ (implicitly in Edixhoven and Yafaev \cite{Edixhoven_Yafaev_2003}, see Remark 3.5 in \cite{Pink} and use the correspondence between Hodge structures and representations of the Deligne torus),
	\item for some choice of $\widetilde{y}$, there is an automorphism $\phi: P \to P$ such that $\widetilde{y} = \phi \circ \widetilde{x}$ (Pink \cite{Pink}), or
	\item for some choice of $\widetilde{y}$, there is a homomorphism $\phi: M \to P$ such that $\widetilde{y} = \phi \circ \widetilde{x}_0$ (Richard and Yafaev \cite{Richard_Yafaev}).
\end{itemize}

Since there is a Shimura subdatum $(M,X_M^+)$ of $(P,X^+)$ such that $\widetilde{x}_0 \in X_M^+$ and since we can always use Lemma 2.2 in \cite{OrrThesis} to adjust the congruence subgroup, the generalized Hecke orbits of $x$ in the sense of Pink and in the sense of Richard and Yafaev are both contained in $\Sigma_{S_{\mathbb{C}}}(x)$.

Furthermore, if $\rho \circ \widetilde{x} = g \cdot (\rho \circ \widetilde{y}) \cdot g^{-1}$ for some $g \in \GL(V)(\mathbb{Q})$, then we must have that $\rho(M) \subset g\rho(P)g^{-1}$. Using this, one can easily prove that the generalized Hecke orbit of $x$ in the sense of Edixhoven and Yafaev is contained in the generalized Hecke orbit in the sense of Richard and Yafaev and hence also in $\Sigma_{S_{\mathbb{C}}}(x)$.

We now want to relate the notion of $\Sigma$-orbits {to} the definition of structures of finite rank which can be found in Definition 1.5 of \cite{Aslanyan}. {One can easily do this} using Lemma \ref{lem:heckeorbitintersection2} and the characterization of special subvarieties of products of modular curves in \cite{Edixhoven_2005}.

Indeed, let $x=(x_1,\dots, x_{n_0})$ be a point in $Y(1)^{n_0}(\mathbb{C})$ and let $Z=Y(1)^n_\mathbb{C}$ for some integers $n_0,n\geq 0$. Then, $\Sigma_Z(x)$ is equal to the structure of finite rank $ (\overline{\Xi})^n$ where $\Xi$ consists of $x_1, \dots , x_{n_0}$ and all CM points of $Y(1)$ and $\overline{\Xi}$ is defined as in \cite{Aslanyan}.

On the other hand, if $ (\overline{\Xi})^n\subset Z=Y(1)^n_\mathbb{C}$ is a structure of finite rank {and} $x_1, \dots, x_{n_0}$ are the non-CM points of $\Xi$, then $ (\overline{\Xi})^n$ is contained in $\Sigma_Z(x)$ where $x=(x_1,\dots, x_{n_0})$.

A notion of structure of finite rank has also been introduced in \cite{AD22} in the setting of connected pure Shimura varieties. This definition is different from the one in \cite{Aslanyan}: in $Y(1)_{\mathbb{C}}^n$, a structure of finite rank in the sense of \cite{Aslanyan} is usually not contained in a structure of finite rank in the sense of \cite{AD22}. Indeed, it can be deduced from Theorems 2.4 and 2.11 in \cite{Richard_Yafaev} that a structure of finite rank in $Y(1)_{\mathbb{C}}^n$ in the sense of \cite{AD22} is contained in the union of all special points of $Y(1)_{\mathbb{C}}^n$ with a finite union of classical Hecke orbits in $Y(1)_{\mathbb{C}}^n$ as defined in \cite{Richard_Yafaev} and so contained in the union of all special points of $Y(1)_{\mathbb{C}}^n$ with a finite union of isogeny classes.
On the other hand, as soon as $n\geq 2$ and $\Xi \subset Y(1)(\mathbb{C})$ contains a non-special point $x$, at most finitely many non-special points, and infinitely many CM points that belong to pairwise distinct isogeny classes, the structure of finite rank $\overline{\Xi}^n$ in $Y(1)_{\mathbb{C}}^n$ in the sense of \cite{Aslanyan} contains points of the form $(x,s, \dots , s)$ {($s\in \Xi$ CM point)}, which are all non-special and belong to infinitely many isogeny classes.

This shows that, in general, given a $\Sigma$-orbit $\Sigma(x)$ and a connected pure Shimura variety $S$, one cannot find a structure of finite rank in the sense of \cite{AD22} containing $\Sigma_{S_\mathbb{C}}(x)$.

On the other hand, by the above considerations it is clear that, for a connected pure Shimura variety $S$, a structure of finite rank in the sense of \cite{AD22} in $S_{\mathbb{C}}$ is contained in $\Sigma_{S_\mathbb{C}}(x)$ for some $\Sigma$-orbit $\Sigma(x)$. 

Finally, let us see that our notion of $\Sigma$-orbit is equivalent to the notion of hybrid orbit by Richard and Yafaev introduced in Definition 6 in \cite{Richard_Yafaev_24a}. Suppose we have a Hodge generic point $x \in X^+_M$ for a connected pure Shimura variety $\mathcal{F}_{\mathrm{pSv}}(M,X^+_M,\Gamma_M)$. The hybrid Hecke orbit of $x$ in $\mathcal{F}_{\mathrm{pSv}}(P,X^+,\Gamma)$ is the image of the set of all $x' \in X^+$ such that $\phi(x^{\text{ad}}) = x'^{\text{ad}}$ for a homomorphism $\phi: M^{\text{ad}} \to M_{x'}^{\text{ad}}$ (that therefore induces a Shimura morphism $(M^{\text{ad}},X_M^{+,\text{ad}}) \to (M_{x'}^{\text{ad}},X_{x'}^{+,\text{ad}})$), where $M_{x'} \subset P$ denotes the Mumford-Tate group of $x'$ and $X^+_{x'} = M_{x'}(\mathbb{R})^+ \cdot x'$. The image in $\Gamma \backslash X^+$ of any such $x'$ will be in the $\Sigma$-orbit of the image of $x$ under $X^+_M \to \Gamma_M \backslash X^+_M$ according to our definition because of Propositions \ref{prop:heckeunderfinitecover} and \ref{prop:heckeunderfinitecover2} since:
\begin{itemize}
\item $(M,X^+_M) \to (M^{\text{ad}},X_M^{+,\text{ad}})$ is a Shimura immersion and
\item $(M_{x'},X^+_{x'}) \to (M_{x'}^{\text{ad}},X_{x'}^{+,\text{ad}})$ is a Shimura immersion.
\end{itemize}

Vice versa, suppose that $x'$ is a pre-image in $X^+$ of a point of our $\Sigma$-orbit of $x$ in $\Gamma \backslash X^+$. It can be deduced from Theorem 7.1 in \cite{Richard_Yafaev_24a} and from Lemma \ref{lem:heckeorbitintersection2} that $x'$ also belongs to the hybrid Hecke orbit of $x$.

\section{$\Sigma$-special subvarieties}\label{sec:sigmaspecial}

\begin{defn}\label{defn:sigmaspecial2}
	For any collection $\Sigma = \{\Sigma_{Z} \subset Z(K);\mbox{ $Z$ distinguished variety}\}$ of sets of closed points of distinguished varieties, a weakly special subvariety $W$ of a distinguished variety $Z$ is called \emph{$\Sigma$-special} if there are distinguished morphisms $\psi:Y\to Z$ and $\phi:Y\to X$ and a point $x\in \Sigma_X$ such that $W$ is an irreducible component of $\psi(\phi^{-1}(x))$.
\end{defn}
In what follows, we are going to consider $\Sigma(x)$-special subvarieties for some $x\in X(K)$ and $X\in \mathfrak{V}$. To lighten the notation, we omit the dependence on $x$. When we consider a $\Sigma$-orbit $\Sigma$, we mean that there are $X\in \mathfrak{V}$ and $x\in X(K)$ such that $\Sigma=\Sigma(x)$ and we use $\Sigma_Z$ to denote $\Sigma_Z(x)$.

\begin{lem}\label{lem:specissigmaspec}
	Let $\Sigma$ be a $\Sigma$-orbit and let $Z$ be a distinguished variety. Any special subvariety of $Z$ is $\Sigma$-special.
\end{lem}

\begin{proof}
	We take $\psi = \id_Z$, $\phi: Z \to X$ equal to the canonical morphism to the final object of the category, and $x \in X(K)$ equal to the unique point of the variety (which is special, being the image of $X(K)$ under the identity morphism) and apply Proposition \ref{prop:heckeandspecialpoint}.
\end{proof}

\begin{lem}\label{lem:sigmaspecial}
	Suppose that $\Sigma$ is a $\Sigma$-orbit. Let $W$ be a $\Sigma$-special subvariety of a distinguished variety $Z$. Then there are distinguished morphisms $\psi:Y\to Z$ and $\phi:Y\to X$ and a point $x\in \Sigma_X$ such that $W = \psi(\phi^{-1}(x))$.
\end{lem}

\begin{proof}
	Follow the proof of Lemma 6.6 in \cite{BD2} and apply Proposition \ref{prop:heckeunderfinitecover}.
\end{proof}

\begin{lem}\label{lem:sigmaspecialchar}
	Suppose that $\Sigma$ is the $\Sigma$-orbit of $x \in X(K)$ for some distinguished variety $X$. Let $W$ be a subvariety of a distinguished variety $Z$. The following are equivalent:
	\begin{enumerate}
		\item $W$ is $\Sigma$-special,
		\item $W \times \{x\}$ is an irreducible component of the intersection of $Z \times \{x\}$ with a special subvariety of $Z \times X$, and
		\item there are distinguished morphisms $\psi:Y\to Z$ and $\phi:Y\to X$ such that $W$ is an irreducible component of $\psi(\phi^{-1}(x))$.
	\end{enumerate}
\end{lem}

\begin{proof}
	It follows directly from the definition of being $\Sigma$-special that (3) implies (1).
	
	Suppose that (2) holds. Then there exists a distinguished morphism $\chi: Y \to Z \times X$ such that $W \times \{x\}$ is an irreducible component of $\chi(Y) \cap (Z \times \{x\})$. Let $\phi$ and $\psi$ denote the compositions of $\chi$ with the projections to $X$ and $Z$ respectively, then
	\[ \chi(Y) \cap (Z \times \{x\}) = \psi(\phi^{-1}(x)) \times \{x\}.\]
	Hence $W$ is an irreducible component of $\psi(\phi^{-1}(x))$ and (3) holds.
	
	Finally, we want to show that (1) implies (2). So suppose that (1) holds. By Lemma \ref{lem:sigmaspecial}, there exist distinguished morphisms $\phi: Y \to Z'$ and $\psi: Y \to Z$ as well as a point $z' \in \Sigma_{Z'}(x)$ such that $W = \psi(\phi^{-1}(z'))$. Unravelling the definition of $\Sigma_{Z'}(x)$, we obtain distinguished morphisms $\phi': Y' \to X$ and $\psi': Y' \to Z'$ as well as a point $y' \in Y'(K)$ such that $\phi'$ has finite fibers, $\phi'(y') = x$, and $\psi'(y') = z'$.

	We consider the fibered product $Y \times_{Z'} Y'$. We apply \ref{ax:2} and Lemma \ref{lem:finfiba2} to find distinguished morphisms $\chi_i: X_i \to Y\times Y'$, $i=1, \dots ,n$, with finite fibers such that $\bigcup_{i=1}^{n}{\chi_i(X_i)} = (Y \times_{Z'} Y')_{\mathrm{red}}$. There is an irreducible component $W_1$ of $\phi^{-1}(z')$ such that $W = \psi(W_1)$. We have $W_1 \times \{y'\} \subset (Y \times_{Z'} Y')_{\mathrm{red}}$. We can assume without loss of generality that $W_1 \times \{y'\} \subset \chi_1(X_1)$ and therefore $W \times \{x\} \subset (Z \times \{x\}) \cap ((\psi \times \phi') \circ \chi_1)(X_1)$.

	Let $W'$ be a subvariety of $Z$ such that $W \subset W'$ and $W' \times \{x\}$ is an irreducible component of $(Z \times \{x\}) \cap ((\psi \times \phi') \circ \chi_1)(X_1)$. We want to show that $W' = W$. Since $\chi_1(X_1) \subset (Y \times_{Z'} Y')_{\mathrm{red}}$, we have that
	\begin{align*} \chi_1^{-1}(W_1 \times \{y'\}) \subset ((\psi \times \phi') \circ \chi_1)^{-1}(W \times \{x\}) \subset ((\psi \times \phi') \circ \chi_1)^{-1}(W' \times \{x\}) \\
	\subset \chi_1^{-1}\left(\bigcup_{\phi'(y'') = x}{\phi^{-1}(\psi'(y'')) \times \{y''\}}\right) = \bigcup_{\phi'(y'') = x}{\chi_1^{-1}\left(\phi^{-1}(\psi'(y'')) \times \{y''\}\right)},
	\end{align*}
	where the union on the right-hand side is finite since $\phi'$ has finite fibers. Furthermore, it follows from \ref{ax:4} and the fact that $\chi_1$ has finite fibers that every irreducible component of the right-hand side has dimension equal to at most $\dim \phi^{-1}(z') = \dim W_1$. Let $W''$ be an irreducible component of $((\psi \times \phi') \circ \chi_1)^{-1}(W' \times \{x\})$ such that $((\psi \times \phi') \circ \chi_1)(W'') = W' \times \{x\}$. It follows from Lemma 7.6 in \cite{BD2}, applied to $(\psi \times \phi') \circ \chi_1: X_1 \to ((\psi \times \phi') \circ \chi_1)(X_1)$, that
	\begin{align*}
		\dim W' = \dim W' \times \{x\} = \dim W'' - (\dim X_1 - \dim ((\psi \times \phi') \circ \chi_1)(X_1)) \\
		\leq \dim W_1 - (\dim X_1 - \dim ((\psi \times \phi') \circ \chi_1)(X_1)).
	\end{align*}
	At the same time, there exists an irreducible component of $\chi_1^{-1}(W_1 \times \{y'\}) \subset ((\psi \times \phi') \circ \chi_1)^{-1}(W \times \{x\})$ that surjects onto $W_1 \times \{y'\}$. Since $\chi_1$ has finite fibers, this irreducible component must have dimension equal to $\dim W_1$ and another application of Lemma 7.6 in \cite{BD2} shows that
	\[ \dim W_1 - (\dim X_1 - \dim ((\psi \times \phi') \circ \chi_1)(X_1)) \leq \dim(W \times \{x\}) = \dim W.\]
	Putting the two inequalities together, we deduce that $\dim W' \leq \dim W$. Since $W \subset W'$ and both $W$ and $W'$ are irreducible, it follows that $W = W'$ as desired.
\end{proof}

\begin{defn}\label{defn:sigmaspecial1}
	For any set $\Sigma_Z$ of closed points of a distinguished variety $Z$, a weakly special subvariety of $Z$ is called \emph{$\Sigma_Z$-point-special} if it contains a point of $\Sigma_Z$.
\end{defn}

Given a $\Sigma $-orbit $\Sigma$ and a distinguished variety $Z$, we want to relate the concepts of a $\Sigma$-special subvariety and of a $\Sigma_Z$-point-special subvariety of $Z$. 

The following lemma is a consequence of Lemmas \ref{lem:sigmaspecialchar} and \ref{lem:heckeorbitintersection2}.

\begin{lem}\label{lem:sigmaspecialpoint}
Let $\Sigma$ be a $\Sigma$-orbit and let $z$ be a closed point of a distinguished variety $Z$. Then, $\{z\}$ is $\Sigma_Z$-point-special if and only if it is $\Sigma$-special.
\end{lem}

\begin{defn}Let $\Sigma$ be a $\Sigma$-orbit and let $Z$ be a distinguished variety. We say that a $z\in Z(K)$ is a \emph{$\Sigma$-special point} if the singleton $\{z\}$ is a $\Sigma$-special subvariety of $Z$.
\end{defn}

Note that, if $\Sigma=\Sigma(x)$ for some $x\in X(K)$, then the $\Sigma$-special points of $Z$ are exactly the points in $\Sigma_Z(x)$, by Lemma \ref{lem:sigmaspecialpoint}.

\begin{lem}\label{lem:finfibweaklyspecial}
	Suppose that the ambient distinguished category satisfies \ref{ax:5}. Let $W$ be a weakly special subvariety of a distinguished variety $Z$. Then there exist distinguished morphisms $\psi:Y\to Z$ and $\phi:Y\to X$ as well as a point $x\in X(K)$ such that $W = \psi(\phi^{-1}(x))$ and $\psi$ has finite fibers.
\end{lem}

\begin{proof}
	We apply \ref{ax:5} to $W$, which is of course weakly optimal for itself in $Z$. It follows that there exist distinguished morphisms $\psi': Y' \to Z$ and $\phi': Y' \to X'$ as well as a point $x' \in X'(K)$ such that $\psi'$ has finite fibers and $W$ is an irreducible component of $\phi'(\psi'^{-1}(x'))$. We can then follow the proof of Lemma 6.6 in \cite{BD2}.
\end{proof}

\begin{rmk}
For all results in this section where we assume that \ref{ax:5} holds, it is actually sufficient to assume that the conclusion of Lemma \ref{lem:finfibweaklyspecial} is satisfied, which of course is much weaker than \ref{ax:5} itself. For instance, in the distinguished category $\mathfrak{C}_{\mathrm{add}}$ whose objects are the powers of the additive group $\mathbb{G}_{a,K}$, \ref{ax:5} does not hold (see Section 9 of \cite{BD2}), but the conclusion of Lemma \ref{lem:finfibweaklyspecial} does.

\end{rmk}

\begin{lem}\label{lem:specialimpliespointspecial}
	Suppose that $\Sigma$ is a $\Sigma$-orbit in a distinguished category that satisfies \ref{ax:5}. If $W$ is a subvariety of a distinguished variety $Z$ that is $\Sigma_Z$-point-special, then $W$ is also $\Sigma$-special.  
\end{lem}

\begin{proof}	
	We apply Lemma \ref{lem:finfibweaklyspecial} to find distinguished morphisms $\psi:Y\to Z$ and $\phi:Y\to X$ as well as a point $x\in X(K)$ such that $W = \psi(\phi^{-1}(x))$ and $\psi$ has finite fibers. Since $W$ is $\Sigma_Z$-point-special, there exists $w \in W(K) \cap \Sigma_Z$. By Proposition \ref{prop:heckeunderfinitecover} and the definition of a $\Sigma$-orbit, any $y$ in the non-empty set $(\psi^{-1}(w) \cap \phi^{-1}(x))(K)$ belongs to $\Sigma_Y$. It follows that $x \in \Sigma_X$ and hence $W$ is $\Sigma$-special.
\end{proof}

\begin{prop}\label{prop:permanencesigmaspecial}
	Suppose that $\Sigma$ is a $\Sigma$-orbit in a distinguished category. Let $\phi: Y \to X$ be a distinguished morphism and let $U$ and $Z$ be two distinguished varieties. Then the following hold:
	\begin{enumerate}
		\item If $W$ is a $\Sigma$-special subvariety of $Y$, then $\phi(W)$ is a $\Sigma$-special subvariety of $X$. Similarly, if $W$ is a $\Sigma_Y$-point-special subvariety of $Y$, then $\phi(W)$ is a $\Sigma_X$-point-special subvariety of $X$. 
		\item If $W$ is a $\Sigma$-special subvariety of $U$ and $W'$ is a $\Sigma$-special subvariety of $Z$, then $W \times W'$ is a $\Sigma$-special subvariety of $U \times Z$. Similarly, if $W$ is a $\Sigma_U$-point-special subvariety of $U$ and $W'$ is a $\Sigma_Z$-point-special subvariety of $Z$, then $W \times W'$ is a $\Sigma_{U \times Z}$-point-special subvariety of $U \times Z$.
		\item If $W$ is a $\Sigma$-special subvariety of $X$, every irreducible component of $\phi^{-1}(W)$ is a $\Sigma$-special subvariety of $Y$. If $\phi$ has finite fibers and $W$ is a $\Sigma_X$-point-special subvariety of $X$, every irreducible component of $\phi^{-1}(W)$ of dimension $\dim W$ is a $\Sigma_Y$-point-special subvariety of $Y$. 
	\end{enumerate}
\end{prop}

\begin{proof}
	Part (1) directly follows from Lemma \ref{lem:sigmaspecial} and Lemma 6.7 in \cite{BD2}.
	
	Part (2) is clear from Proposition \ref{prop:heckeunderproduct} and the definitions since the irreducible components of a product are precisely all possible products of an irreducible component of the first factor and an irreducible component of the second factor.
	
	Part (3) can be deduced from Proposition \ref{prop:heckeunderfinitecover} and Lemma 6.7 in \cite{BD2} by noting that the proof of Lemma 6.7 in \cite{BD2} together with Lemma \ref{lem:sigmaspecial} implies that irreducible components of pre-images of $\Sigma$-special subvarieties under distinguished morphisms are $\Sigma$-special. Furthermore, every irreducible component of $\phi^{-1}(W)$ of dimension $\dim W$ surjects onto $W$ under $\phi$.
\end{proof}

\begin{cor}\label{cor:sigmaspeccont}
	Suppose that $\Sigma$ is a $\Sigma$-orbit in a distinguished category satisfying \ref{ax:5} and let $Z$ be a distinguished variety. Let $W\subset V$ be weakly special subvarieties of $Z$ and suppose $W$ is $\Sigma$-special. Then $V$ is $\Sigma$-special.
\end{cor}

\begin{proof}By Lemma \ref{lem:finfibweaklyspecial}, as $V$ is weakly special, there are distinguished morphisms $\phi:Y\to Z$ and $\psi :Y\to X$ and a point $x\in X(K)$ such that $V=\phi(\psi^{-1}(x))$ and $\phi$ has finite fibers. Because of the latter fact, components of $\phi^{-1}(V) $ have dimension $\leq \dim V$ and clearly we must have $\psi^{-1}(x)\subset \phi^{-1}(V) $. Therefore, also the components of $\psi^{-1}(x)$ have dimension $\leq \dim V$ and there must be one of dimension $\dim V$. This means that  $\phi^{-1}(V) $ and $\psi^{-1}(x)$  have a common component of dimension $\dim V$, which we call $V'$. Note that $\phi(V')=V$ because $\phi(V')$ is a weakly special subvariety of the same dimension as $V$.
	
We pick a component $W'$ of $\phi|_{V'}^{-1}(W)$ such that $\phi(W') = W$. This is a component of $\phi^{-1}(W)$ because $\phi $ has finite fibers. Then $W'$ is $\Sigma$-special and the same holds for $\{x\}$, by Proposition \ref{prop:permanencesigmaspecial}. Another application of the same proposition gives the desired claim.
\end{proof}

\section{Examples II}\label{sec:exii}

The main objective of this section is to give a description of $\Sigma$-special and $\Sigma$-point-special subvarieties in some specific distinguished categories.

\subsection{Semiabelian varieties}

Let $G$ be a semiabelian variety. We want to see how the concepts of a $\Sigma$-special subvariety and of a $\Sigma_G$-point-special subvariety of $G$ are related in this context. Since the distinguished category of semiabelian varieties satisfies \ref{ax:5} by Theorem 9.3 in \cite{BD2}, we know already that a $\Sigma_G$-point-special subvariety of $G$ is $\Sigma$-special. 

We call a semiabelian variety \emph{almost-split} if it is isogenous to a split semiabelian variety, i.e., to a semiabelian variety of the form $A \times \mathbb{G}^n_{m,K}$ for some abelian variety $A$ over $K$ and some non-negative integer $n$. One can show that the full subcategory of almost-split semiabelian varieties within the distinguished category of semiabelian varieties is also distinguished.

\begin{prop}\label{prop:almostsplit}
	In the distinguished category of almost-split semiabelian varieties, the following holds:
	let $\Sigma$ be the $\Sigma$-orbit of $x\in G_0(K)$. Then, a $\Sigma$-special subvariety of a distinguished variety $G$ contains a $\Sigma$-special point and is therefore $\Sigma_G$-point-special.
\end{prop}

\begin{proof}
	We can and will assume that $x$ does not belong to any proper special subvariety of $G_0$ as done in Subsection \ref{subsec:semiab1}.
	
	Let $W$ be a $\Sigma$-special subvariety of $G$. Then, $W$ is a component of $\psi(\phi^{-1}(x))$ for some distinguished morphisms $\psi:Y\to G$ and $\phi:Y\to G_0$. Since $x$ does not belong to any proper special subvariety of $G_0$, we have that $\phi $ is surjective. We may assume as in Subsection \ref{subsec:semiab1} that $\phi$ is a homomorphism of algebraic groups. Since $Y$ is almost-split, there exist an abelian variety $B$ over $K$, an integer $n \geq 0$, and an isogeny $\omega: B \times \mathbb{G}^n_{m,K}\to Y$. The kernel of $\phi \circ \omega$ is of the form $H \times H'$ for algebraic subgroups $H$ of $B$ and $H'$ of $\mathbb{G}^n_{m,K}$ respectively. It then follows from Poincar\'e's complete reducibility theorem (see, e.g., \citep[Theorem 1, p.~173]{Mumford}) that there is a semiabelian subvariety of $B \times \mathbb{G}^n_{m,K}$ such that $\phi \circ \omega$ restricts to an isogeny from this semiabelian subvariety to $G_0$. Therefore, there also exists a semiabelian subvariety of $Y$ such that $\phi$ restricts to an isogeny from this semiabelian subvariety to $G_0$. By Lemma 7.3/5 on p. 180 of \cite{Neronmodels}, we can conclude that there exists a homomorphism of algebraic groups $\phi':G_0\to Y$ such that $\chi=\phi\circ \phi'$ is the multiplication by some positive integer $d$. Then, $\chi^{-1}(x)$ is a finite set of points which lie in $\Sigma_{G_0} $ by Proposition \ref{prop:permanencesigmaspecial} and some component of $\phi^{-1}(x)$ must contain a point of $\phi'(\chi^{-1}(x))$ which must be a point of $\Sigma_Y$ by the same Proposition \ref{prop:permanencesigmaspecial}. But the components of $\phi^{-1}(x)$ are all translates of each other by torsion points of $Y$ and since translating by a torsion point is a distinguished morphism, each of them contains a point of $\Sigma_Y$ by Proposition \ref{prop:permanencesigmaspecial}. Applying Proposition \ref{prop:permanencesigmaspecial} once more, we deduce that any component of $\psi(\phi^{-1}(x))$ contains a point of $\Sigma_G$.
\end{proof}

\begin{rmk}
	We now see that, if the almost-split condition is dropped, meaning we consider the distinguished category of semiabelian varieties, the statement of the above proposition cannot hold.
	
	Let $G_0$ be an arbitrary semiabelian variety and let $x_0 \in G_0(K)$. Let $E$ be an elliptic curve with $\mathrm{End}(E) \simeq \mathbb{Z}$ and let $y_0 \in E(K)$ be a non-torsion point.
	
	For any abelian variety $A$, we denote its dual abelian variety by $\widehat{A}$. The extensions of $A$ by $\mathbb{G}_{m,K}^n$ over $K$ are classified by $\widehat{A}^n(K)$, see Section 1.1 of \cite{Lars}. Let $G$ be an extension of $E$ by $\mathbb{G}_{m,K}$ that corresponds to $g \in \widehat{E}(K)$.
	
	We have the canonical projection $\pi:G\to E$ and $\pi^{-1}(y_0)$ is a $\Sigma(x_0,y_0)$-special subvariety of $G$. We are going to see that, for ``most choices'' of $g$,  $\pi^{-1}(y_0)$ is not $\Sigma_G(x_0,y_0)$-point-special.
	
	Suppose that $\phi: G_0' \to G_0 \times E$ is a homomorphism of semiabelian varieties with finite kernel. Let $A_0$ and $A_0'$ denote the abelian quotients of $G_0$ and $G_0'$ respectively and let $g_0 \in \widehat{A_0}^m(K)$ and $g_0' \in \widehat{A_0'}^{m'}(K)$ denote the elements corresponding to $G_0$ and $G_0'$ respectively. It follows from Section 1 of \cite{Lars} that $g_0' \in (\Gamma'^{\mathrm{div}})^{m'}$ where $\Gamma' = \left\langle \mathrm{Hom}(\widehat{A_0 \times E},\widehat{A_0'}) \cdot (g_0,0_{\widehat{E}}) \right\rangle$ is a finitely generated subgroup of $\widehat{A_0'}(K)$. Recall that $\Gamma'^{\mathrm{div}}$ indicates the division group of $\Gamma'$ in $\widehat{A_0'}(K)$. Suppose now that $\psi: G_0' \to G$ is an arbitrary homomorphism of semiabelian varieties. It follows again from Section 1 of \cite{Lars} that either $\pi \circ \psi$ is trivial or $\chi(g) \in \Gamma'^{\mathrm{div}}$ for some homomorphism $\chi: \widehat{E} \to \widehat{A_0'}$ with finite kernel. Since $\mathrm{End}(\widehat{E}) \simeq \mathrm{End}(E) \simeq \mathbb{Z}$, we find in the latter case that $g \in \Gamma^{\mathrm{div}}$ where $\Gamma = \left\langle \mathrm{Hom}(\widehat{A_0 \times E},\widehat{E}) \cdot (g_0,0_{\widehat{E}}) \right\rangle$. 
	
	Hence, if we choose $G$ such that $g \not\in \Gamma^{\mathrm{div}}$ (this can be done by work of Frey and Jarden \cite{FreyJarden}), then every $\Sigma(x_0,y_0)$-special point of $G$ must project to a torsion point of $E$ under $\pi$. It follows that $\pi^{-1}(y_0)$ cannot be $\Sigma_G(x_0,y_0)$-point-special.
\end{rmk}

We conclude this subsection by showing that, even if a subvariety $V$ of $G$ is $\Sigma$-special and not $\Sigma_G$-point-special, it is possible to make $\Sigma$ a bit larger (independently of $V$) so that $V$ contains a point of $\Sigma_G$. On the other hand, the above remark shows that it is not possible to do this uniformly in $G$ in general.

\begin{prop} In the distinguished category of semiabelian varieties, the following holds:
	let $G_0$ be a semiabelian variety, let $\Sigma$ be the $\Sigma$-orbit of $x\in G_0(K)$, and let $G$ be another fixed semiabelian variety. There exist $\gamma_1', \hdots, \gamma_n' \in G(K)$ such that, for $\Sigma' = \Sigma(x,\gamma_1',\hdots,\gamma_n')$, we have that every $\Sigma'$-special subvariety of $G$ is $\Sigma'_G$-point-special.
\end{prop}

\begin{proof}
	Let $\pi_A: G \to A$ denote the maximal abelian quotient of $G$ and set $T := \ker \pi_A$.
	
	We first claim that, for an arbitrary $\Sigma$-orbit $\widetilde{\Sigma}$, every $\widetilde{\Sigma}$-special subvariety of $A$ is $\widetilde{\Sigma}_A$-point-special and every $\widetilde{\Sigma}$-special subvariety of $T$ is $\widetilde{\Sigma}_T$-point-special. Indeed, suppose that $\widetilde{\Sigma}={\Sigma}(\widetilde{x})$ for some $\widetilde{x} \in \widetilde{G}_0(K)$ and that $W \subset T$ is $\widetilde{\Sigma}$-special. Because of Proposition \ref{prop:heckeunderfinitecover2}, we can and will assume without loss of generality that $\widetilde{x}$ does not belong to any proper special subvariety of $\widetilde{G}_0$.  By Lemma \ref{lem:sigmaspecialchar}, there exist distinguished morphisms $\phi: Y\to \widetilde{G}_0$ and $\psi: Y \to T$ such that $W$ is an irreducible component of $\psi(\phi^{-1}(\widetilde{x}))$. As $\widetilde{x}$ does not belong to any proper special subvariety of $\widetilde{G}_0$, the morphism $\phi$ is surjective. We can and will assume that $\phi$ is a homomorphism of algebraic groups. Let $\pi': Y \to T_Y$ denote the maximal torus quotient of $Y$, which exists because of the Rosenlicht decomposition of $Y$ (see Proposition 3.1 in \cite{Brion09}), and let $\widetilde{\pi}: \widetilde{G}_0 \to \widetilde{T}_0$ denote the maximal torus quotient of $\widetilde{G}_0$. By maximality, there exist a distinguished morphism $\chi: T_Y \to T$ and a homomorphism of algebraic groups $\omega: T_Y \to \widetilde{T}_0$ such that $\psi = \chi \circ \pi'$ and $\widetilde{\pi} \circ \phi = \omega \circ \pi'$. Furthermore, the surjectivity of $\phi$ implies that $Y/(\ker \phi + \ker \pi')$ is a torus quotient of $\widetilde{G}_0$ and the maximality of $\widetilde{T}_0$ tells us that $Y/(\ker \phi + \ker \pi')$ is a quotient of $ \widetilde{T}_0$. Thus,  $  \ker(\widetilde{\pi} \circ \phi) \subset \ker \phi + \ker \pi'$. The other inclusion is trivial and we have $\ker \phi + \ker \pi' = \ker(\widetilde{\pi} \circ \phi) = \ker(\omega \circ \pi')$. We deduce that
	\begin{align*}
		\psi(\phi^{-1}(\widetilde{x})) = \chi(\pi'(\phi^{-1}(\widetilde{x}))) = \chi(\pi'(\ker \pi' + \phi^{-1}(\widetilde{x}))) \\
		= \chi(\pi'((\omega \circ \pi')^{-1}(\widetilde{\pi}(x)))) = \chi(\omega^{-1}(\widetilde{\pi}(x))).
	\end{align*}
	It now follows as in the proof of Proposition \ref{prop:almostsplit} that $W$ is $\widetilde{\Sigma}_T$-point-special. A completely analogous argument shows that every $\widetilde{\Sigma}$-special subvariety of $A$ is $\widetilde{\Sigma}_A$-point-special.
	
	Let now $\pi_0:G_0\to A_0$ be the maximal abelian quotient of $G_0$ and $y=\pi_0(x)$. Because of Proposition \ref{prop:heckeunderfinitecover2}, we can and will assume without loss of generality that $x$ does not belong to any proper special subvariety of $G_0$.
	We have seen in Subsection \ref{subsec:semiab1} that there are generators $\psi_1, \dots ,\psi_n$ of $\Hom(G_0,A)$ such that $\Sigma_A$ is equal to $\langle \gamma_1,\dots, \gamma_n \rangle^{\mathrm{div}}$ where $\gamma_i=\psi_i(x) \in A(K)$ for all $i=1,\dots, n$. As all elements of $\Hom(G_0,A)$ factor through $A_0$, there are $\phi_1, \dots ,\phi_n\in \Hom(A_0,A)$ such that $\psi_i=\phi_i\circ \pi_0$ for all $i=1,\dots, n$. Therefore, $\gamma_i = \phi_i(y)$ for all $i=1,\dots, n$.

	 For each $i = 1,\hdots,n$, fix $\gamma'_i \in G(K)$ such that $\pi_A(\gamma'_i) = \gamma_i$. We claim that the conclusion of the proposition holds with $\Sigma'=\Sigma(x') $ for $x'=(x,\gamma'_1,\hdots,\gamma'_n) \in (G_0 \times G^n)(K)$. Note that every $\Sigma$-special subvariety is $\Sigma'$-special by Lemma \ref{lem:sigmaspecialchar}; in particular, we have that $\Sigma_Z \subset \Sigma'_Z$ for every semiabelian variety $Z$.
	 We have just proved our first claim.
	 
	 Now we want to show that $\Sigma'_A\subset \Sigma_A$. Let $H$ be the semiabelian subvariety of $G_0 \times G^n$ such that $\langle \{(x,\gamma'_1,\hdots,\gamma'_n) \}\rangle = (t_0,t_1,\hdots,t_n) + H \subset G_0 \times G^n$ (we recall that $\langle \cdot \rangle$ denotes the special closure) for some torsion points $t_1,\hdots,t_n$ of $G$ and a torsion point $t_0$ of $G_0$. Let $B \subset A_0 \times A^n$ denote the maximal abelian quotient of $H$. It follows that 
	 	\[B = \langle \{(y-\pi_0(t_0),\phi_1(y)-\pi_A(t_1),\hdots,\phi_n(y)-\pi_A(t_n))\} \rangle.\]
 	Since $G_0 = \langle \{x\} \rangle$ and so $A_0 = \langle \{y\} \rangle$, it follows that
 	\[(z-\pi_0(t_0),\phi_1(z)-\pi_A(t_1),\hdots,\phi_n(z)-\pi_A(t_n)) \in B(K) \quad \forall z \in A_0(K).\]
 	Since $B$ is an abelian subvariety of $A_0\times A^n$, it follows that $(\id_{A_0},\phi_1,\hdots,\phi_n)(A_0) \subset B$. But $(y,\phi_1(y),\hdots,\phi_n(y)) \in (\id_{A_0},\phi_1,\hdots,\phi_n)(A_0)(K)$ and so $B = (\id_{A_0},\phi_1,\hdots,\phi_n)(A_0)$.
	
 	Combining this with the fact that $A_0$ and $A$ are the maximal abelian quotients of $G_0$ and $G$ respectively as well as with arguments from Subsection \ref{subsec:semiab1}, we get
	\begin{align*}
		\Sigma'_A&=\langle \Hom(H, A)\cdot (x,\gamma_1'-t_1, \dots, \gamma_n'-t_n ) \rangle^{\mathrm{div}}\\ &=\langle \Hom(B, A)\cdot (y,\phi_1(y)-\pi_A(t_1), \dots,\phi_n(y)-\pi_A(t_n) ) \rangle^{\mathrm{div}}\\
		& \subset
		\langle \Hom(A_0, A)\cdot y \rangle^{\mathrm{div}}\subset \langle \Hom(G_0, A)\cdot x\rangle^{\mathrm{div}}
		=\Sigma_A.
	\end{align*}

	  Now, for every $\Sigma'$-special subvariety $W$ of $G$, the subvariety $\pi_A(W)$ of $A$ is $\Sigma'_A$-point-special by the above.
	   Hence, by construction of $\Sigma'$, there is a point $\gamma'$ in $\Sigma'_G$ such that $(W-\gamma') \cap T \neq \emptyset$. By Proposition \ref{prop:permanencesigmaspecial}(1)-(2), the subvarieties $W \times \{\gamma'\} \subset G^2$ and $W-\gamma' \subset G$ are both $\Sigma'$-special. Furthermore, every component of $(W-\gamma') \cap T$, regarded as a subvariety of $T$, is $\Sigma'$-special by Proposition \ref{prop:permanencesigmaspecial}(3) and therefore $\Sigma'_T$-point-special by the above. As $\Sigma'_T \subset \Sigma'_G$ and $\Sigma'_G + \Sigma'_G \subset \Sigma'_G$ by Proposition \ref{prop:permanencesigmaspecial}, we conclude that $W$ contains a point of $\Sigma'_G$.
\end{proof}

\subsection{Connected pure Shimura varieties}

In the following Proposition \ref{prop:weaklyspecialshimura} and its proof, we will use connected pure Shimura data and varieties and Shimura morphisms, submersions, immersions, coverings, and embeddings as well as quotient Shimura morphisms, quotient Shimura data, and Shimura subdata as defined in Definitions 3.1, 3.2, and 3.4 in \cite{BD2}. Given a connected pure Shimura datum $(P,X^+)$, we will also use the notation $P(\mathbb{Q})_+$ as introduced in Section 3 in \cite{BD2}.

\begin{prop}\label{prop:weaklyspecialshimura}
	Let $S$ be a connected pure Shimura variety over $K$ and let $W \subset S$ be a weakly special subvariety. Then there exist connected pure Shimura varieties $S'$, $S_1$, and $S_2$ as well as a Shimura embedding $i: S' \to S$, a Shimura covering $p: S' \to S_1 \times S_2$, and a point $w \in S_2(K)$ such that $W$ is an irreducible component of $i(p^{-1}(S_1 \times \{w\}))$.
\end{prop}

\begin{proof}
	Suppose that $S = \mathcal{F}_{\mathrm{pSv}}(P,X^+,\Gamma)$ where the functor $\mathcal{F}_{\mathrm{pSv}}$ is defined as in Section 3 in \cite{BD2}. By Proposition 5.4 in \cite{G17}, there exist connected pure Shimura varieties $S'$ and $\widetilde{S}$, a Shimura embedding $i: S' \to S$, a Shimura quotient morphism $\widetilde{p}: S' \to \widetilde{S}$, and a point $\widetilde{w} \in \widetilde{S}(K)$ such that $W$ is an irreducible component of $i(\widetilde{p}^{-1}(\widetilde{w}))$ (note that Shimura subdata and quotient Shimura data of connected pure Shimura data are again pure by Proposition 2.9 in \cite{G17} and Remark 3.7 in \cite{BD2}). Proposition 5.4 in \cite{G17} holds over $\mathbb{C}$, but we can always find an algebraically closed subfield $L$ of $K$ over which $W$ is defined and such that $L$ embeds into $\mathbb{C}$. Proposition 5.4 in \cite{G17} then yields $\widetilde{w} \in \widetilde{S}(\mathbb{C})$, but since $i$ is finite by Proposition 3.6 in \cite{BD2}, we must have that $\widetilde{w} \in \widetilde{S}(L)$. Since $i$ is a finite morphism 
	and since $\widetilde{p}^{-1}(\widetilde{w})$ is equidimensional by \ref{ax:4}, the above conclusion is equivalent to the fact $(\ast)$ that $\widetilde{p}^{-1}(\widetilde{w})$ and $i^{-1}(W)$ share a common irreducible component that surjects onto $W$ under $i$.
	
	Suppose that $S' = \mathcal{F}_{\mathrm{pSv}}(Q,Y^+,\Delta)$. We can assume without loss of generality that $\widetilde{S} = \mathcal{F}_{\mathrm{pSv}}((Q,Y^+)/N,\widetilde{\Delta})$ for some normal algebraic subgroup $N$ of $Q$ and some congruence subgroup $\widetilde{\Delta} \subset (Q/N)(\mathbb{Q})_+$ (see Remark 3.5 in \cite{BD2}).
	
	Let $Z$ denote the center of $Q$, then we can form the quotient Shimura datum $(Q,Y^+)/Z =: (Q^{\mathrm{ad}},Y^{+,\mathrm{ad}})$ as well as $(Q,Y^+)/ZN$ (see again Remark 3.5 in \cite{BD2}). Since $Z$ is of multiplicative type by \cite[Corollary 17.62(a)]{MilneAG}, the two quotient Shimura morphisms $(Q,Y^+) \to (Q^{\mathrm{ad}},Y^{+,\mathrm{ad}})$ and $(Q,Y^+)/N \to (Q,Y^+)/ZN$ are Shimura immersions and therefore also Shimura coverings.

	We now apply Lemma 2.2 in \cite{OrrThesis} to find a congruence subgroup $\Delta'$ of $(Q/(ZN))(\mathbb{Q})$ that contains the image of $\widetilde{\Delta}$ under the map $(Q/N)(\mathbb{Q}) \to (Q/(ZN))(\mathbb{Q})$. By Proposition 2.2 in \cite{RohlfsSchwermer} we may assume that $\Delta'$ is contained in $(Q/(ZN))(\mathbb{Q})_+$. We replace $N$ by $ZN$, $\widetilde{\Delta}$ by $\Delta'$, and $\widetilde{w}$ by its image under the induced Shimura morphism: thanks to \ref{ax:4} and the fact that $(Q,Y^+)/N \to (Q,Y^+)/ZN$ is a Shimura covering, this does not affect $(\ast)$ or any other of the properties stated above. Having done that, we have that $Z \subset N$.
	
	Furthermore, by Theorem 21.51 in \cite{MilneAG}, we have that $Q^{\mathrm{ad}} = Q_1 \times Q_2$ for some algebraic subgroups $Q_1, Q_2$ of $Q^{\mathrm{ad}}$ such that $N/Z = Q_1 \times \{1\}$. For $\{i,j\} = \{1,2\}$, we define connected pure Shimura data $(Q_i,Y_i^+)$ by $(Q_i,Y_i^+) := (Q^{\mathrm{ad}},Y^{+,\mathrm{ad}})/Q_j$. Note that $(Q_2,Y_2^+) \simeq (Q,Y^+)/N$ and we will from now on identify these two connected pure Shimura data via the canonical isomorphism. We obtain a Shimura covering $(Q^{\mathrm{ad}},Y^{+,\mathrm{ad}}) \to (Q_1,Y_1^+) \times (Q_2,Y_2^+)$.
	
	Let $\mathrm{ad}: Q \to Q^{\mathrm{ad}}$ and $\psi: Q^{\mathrm{ad}} \to Q/N = Q_2$ denote the quotient homomorphisms. Set $\Delta_2 = \widetilde{\Delta}$ regarded as a congruence subgroup of $Q_2(\mathbb{Q})_+$. Fix any congruence subgroup $\Delta_1 \subset Q_1(\mathbb{Q})_+$ (which exists, see Proposition 2.2 in \cite{RohlfsSchwermer}). Then $\Delta_1 \times \Delta_2$ is a congruence subgroup of $Q^{\mathrm{ad}}(\mathbb{Q})_+$. Note that replacing $\Delta$ by some other congruence subgroup contained in $(\psi \circ \mathrm{ad})^{-1}(\Delta_2)$ does not affect $(\ast)$ or any other of the properties stated above. We can therefore again apply Lemma 2.2 in \cite{OrrThesis} and assume without loss of generality that $\Delta \subset \ad^{-1}(\Delta_1 \times \Delta_2)$.
	
	We now set $S_i = \mathcal{F}_{\mathrm{pSv}}(Q_i,Y_i^+,\Delta_i)$ ($i = 1,2$) and $p$ equal to the image under $\mathcal{F}_{\mathrm{pSv}}$ of the induced Shimura morphism $(Q,Y^+,\Delta) \to (Q^{\mathrm{ad}},Y^{+,\mathrm{ad}},\Delta_1 \times \Delta_2) \to (Q_1,Y_1^+,\Delta_1) \times (Q_2,Y_2^+,\Delta_2)$. We also let $w$ denote the point in $S_2(K)$ that is identified with $\widetilde{w}$ under the identification $((Q,Y^+)/N,\widetilde{\Delta}) = (Q_2,Y_2^+,\Delta_2)$. The proposition now follows since $p^{-1}(S_1 \times \{w\}) = \widetilde{p}^{-1}(\widetilde{w})$.
\end{proof}

\begin{prop}\label{prop:pointShimura}
	 In the distinguished category of connected pure Shimura varieties over $K$, the following holds.
	Let $\Sigma$ be the $\Sigma$-orbit of $x\in S_0(K)$. Then, a $\Sigma$-special subvariety of a connected pure Shimura variety $S$ contains a $\Sigma$-special point and is therefore $\Sigma_S$-point-special.
\end{prop}

\begin{proof}
	Let $W$ be a $\Sigma$-special subvariety of $S$. By Proposition \ref{prop:weaklyspecialshimura}, there exist connected pure Shimura varieties $S'$, $S_1$, and $S_2$ as well as a Shimura embedding $i: S' \to S$, a Shimura covering $p: S' \to S_1 \times S_2$, and a point $w \in S_2(K)$ such that $W$ is an irreducible component of $i(p^{-1}(S_1 \times \{w\}))$. By \ref{ax:4}, all irreducible components of $p^{-1}(S_1 \times \{w\})$ have the same dimension, which must be equal to $\dim W$ as well as to $\dim S_1$ since $i$ and $p$ have finite fibers. It follows as in the proof of Proposition \ref{prop:weaklyspecialshimura} that $i^{-1}(W)$ and $p^{-1}(S_1 \times \{w\})$ share an irreducible component. By Proposition \ref{prop:permanencesigmaspecial}(3), this component is $\Sigma$-special in $S'$, which implies by Proposition \ref{prop:permanencesigmaspecial}(1) that $\{w\}$ is $\Sigma$-special in $S_2$. By Lemma \ref{lem:sigmaspecialpoint}, we have that $w \in \Sigma_{S_2}$. By Lemma 13.3 in \cite{MilneISV}, whose proof can be modified to work also with our definition of a connected Shimura datum that is slightly different from the one in \cite{MilneISV} (see Remark 4.4 in \cite{BD2}), there exists a special point $s_1 \in S_1(K)$. By Propositions \ref{prop:heckeunderproduct} and \ref{prop:heckeandspecialpoint}, we have that $(s_1,w) \in \Sigma_{S_1 \times S_2}$. So $S_1 \times \{w\}$ is $\Sigma_{S_1 \times S_2}$-point-special. By Proposition \ref{prop:permanencesigmaspecial}(3), every irreducible component of $p^{-1}(S_1 \times \{w\})$ is $\Sigma_{S'}$-point-special since $p$ is a Shimura covering, so has finite fibers and all irreducible components of $p^{-1}(S_1 \times \{w\})$ are of dimension $\dim S_1$. By Proposition \ref{prop:permanencesigmaspecial}(1) and the above, $W$ is therefore $\Sigma_S$-point-special and we are done.
\end{proof}

Our notion of a $\Sigma$-special subvariety of a connected pure Shimura variety $S$ (which has now turned out to coincide with the notion of ``$\Sigma_S$-point-special") is related to the sets $WS_X(M,X_M,x)$ defined in \cite{Richard_Yafaev_24a}, see Section 5 and Theorem 7.1 in \cite{Richard_Yafaev_24a} and Lemma \ref{lem:sigmaspecialchar}.

\section{Mordell-Lang}\label{sec:ml}

We now formulate two statements that, in view of Subsection \ref{subsec:semiab1}, we can consider as analogues of the Mordell-Lang conjecture.

\begin{defn}\label{defn:WML} Fix an integer $m\geq 0$ and let $\Sigma$ be a $\Sigma$-orbit.
	For a distinguished variety $Z$, we say that $\WML(Z,\Sigma,m)$ ($\WML$ stands for \emph{weak Mordell-Lang}) holds if any subvariety of $Z$ of dimension at most $m$ that contains a Zariski dense subset of $\Sigma$-special points is weakly special.
\end{defn}
\begin{defn}\label{defn:SML} Fix an integer $m\geq 0$ and let $\Sigma$ be a $\Sigma$-orbit.
	For a distinguished variety $Z$, we say that $\SML(Z,\Sigma,m)$ ($\SML$ stands for \emph{strong Mordell-Lang}) holds if any subvariety of $Z$ of dimension at most $m$ contains at most finitely many maximal $\Sigma$-special subvarieties.
\end{defn}

It is easy to see that $\SML(Z,\Sigma,m)$ implies $\WML(Z,\Sigma,m)$. If the distinguished category satisfies \ref{ax:5} and $\Sigma$ is a $\Sigma$-orbit, it follows from Lemma \ref{lem:specialimpliespointspecial} and the fact that any weakly special subvariety containing a $\Sigma_Z$-point-special subvariety is itself $\Sigma_Z$-point-special that $\SML(Z,\Sigma,m)$ also implies that every subvariety of $Z$ of dimension at most $m$ contains at most finitely many maximal $\Sigma_Z$-point-special subvarieties.

Under an additional hypothesis that is often satisfied in the applications, the two statements are equivalent.

\begin{lem}\label{lem:wmlequivalentsml}
    Let $\Sigma$ be a $\Sigma$-orbit in a distinguished category that satisfies \ref{ax:5} and let $Z$ be a distinguished variety. If every $\Sigma$-special subvariety of $Z$ contains a dense set of $\Sigma$-special points, then $\WML(Z,\Sigma,m) \Leftrightarrow \SML(Z,\Sigma,m)$ for all non-negative integers $m$.
\end{lem}

\begin{proof}
    One implication has already been discussed above. For the other implication, let $V$ be a subvariety of $Z$ of dimension $\leq m$ and let $S \subset V$ be a maximal $\Sigma$-special subvariety. The Zariski closure of all $\Sigma$-special points in $V$ contains $S$ and is a finite union of weakly special and hence $\Sigma_Z$-point-special subvarieties by $\WML(Z,\Sigma,m)$. By Lemma \ref{lem:specialimpliespointspecial}, each component of this Zariski closure is also $\Sigma$-special. It follows from the maximality of $S$ that $S$ is equal to such a component and so there are at most finitely many possibilities for $S$.
\end{proof}

As one might expect, Zilber-Pink implies Mordell-Lang.

\begin{thm}\label{thm:zpimpliesml}
	Let $\Sigma$ be the $\Sigma$-orbit of $x \in X(K)$ and let $Z$ be a distinguished variety. If the statement $\ZP(Z \times X,m,\dim \langle \{x\} \rangle)$ from Definition 10.1 in \cite{BD2} holds for a non-negative integer $m$, then $\SML(Z,\Sigma,m)$ holds.
\end{thm}
 
\begin{proof}
Let $V$ be a subvariety of $Z$. We claim that a subvariety $V'$ of $V$ is maximal $\Sigma$-special if and only if $V' \times \{x\}$ is optimal of defect $\dim \langle \{x\} \rangle$ for $V \times \{x\}$ in $Z \times X$. This claim clearly implies the theorem.

First of all, note that any subvariety of $V \times \{x\}$ has defect equal to at least $\dim \langle \{x\} \rangle$ since $\langle V_1 \times \{x\} \rangle$ surjects onto $\langle \{x\} \rangle$ for every subvariety $V_1$ of $V$. Furthermore, the defect of a subvariety $V_1 \times \{x\}$ of $Z \times \{x\}$ is equal to $\dim \langle \{x\} \rangle$ if and only if there exists a special subvariety $S$ of $Z \times X$ such that $V_1 \times \{x\} \subset S$ and $\dim S = \dim V_1 + \dim \langle \{x\} \rangle$.

We claim that this is equivalent to the existence of a special subvariety $S$ of $Z \times X$ such that $V_1 \times \{x\}$ is an irreducible component of $(Z \times \{x\}) \cap S$.

For a special subvariety $S$ of $Z \times X$, all fibers of the projection morphism $\pi: S \to X$ are equidimensional by Lemma 7.4 in \cite{BD2} and therefore of dimension $\dim S - \dim \pi(S)$ by the Fiber Dimension Theorem (Theorem 7.3 in \cite{BD2}). So, if $V_1 \times \{x\} \subset S$, then $\dim V_1 \leq \dim S - \dim \pi(S)$. In particular, if $\dim S = \dim V_1 + \dim \langle \{x\} \rangle$, then $\dim \pi(S) \leq \dim \langle \{x\} \rangle$ and hence $\pi(S) = \langle \{x\} \rangle$. Furthermore, in that case, $V_1 \times \{x\}$ must be an irreducible component of $\pi^{-1}(x)$, i.e., an irreducible component of $(Z \times \{x\}) \cap S$.

Vice versa, if $V_1 \times \{x\}$ is an irreducible component of $(Z \times \{x\}) \cap S$ for a special subvariety $S$ of $Z \times X$, then, thanks to Lemma 6.7 in \cite{BD2}, we can assume that $S$ projects onto $\langle \{x\} \rangle$ in the second coordinate. By the above, this implies that $\delta(V_1 \times \{x\}) \leq \dim S - \dim V_1 = \dim \langle \{x\} \rangle$ and therefore $\delta(V_1 \times \{x\}) = \dim \langle \{x\} \rangle$.

This establishes the above claim.

We deduce from Lemma \ref{lem:sigmaspecialchar} that the defect of a subvariety $V' \times \{x\}$ of $V \times \{x\}$ is equal to $\dim \langle \{x\} \rangle$ if and only if $V'$ is $\Sigma$-special. Since any subvariety of $V \times \{x\}$ has defect equal to at least $\dim \langle \{x\} \rangle$ as we have already shown above, it follows that a subvariety $V' \times \{x\}$ of $V \times \{x\}$ is optimal for $V \times \{x\}$ in $Z \times X$ if and only if $V'$ is a maximal $\Sigma$-special subvariety of $V$, which is the claim we wanted to establish.
\end{proof}

\begin{defn}
Let $\Sigma$ be a $\Sigma$-orbit, let $Z$ be a distinguished variety, and let $V$ be a subvariety of $Z$. A subvariety $W$ of $V$ is called a \emph{$\Sigma$-subvariety (of $V$ in $Z$)} if $\langle W \rangle_{\ws}$ is $\Sigma$-special. A subvariety $W$ of $V$ is called \emph{$\Sigma$-atypical (for $V$ in $Z$)} if $W$ is a $\Sigma$-subvariety and $\delta_{\ws}(W) < \dim Z - \dim V$.
\end{defn}

\begin{prop}\label{prop:wo}
Let $\Sigma$ be the $\Sigma$-orbit of $x \in X(K)$, let $Z$ be a distinguished variety, and let $V$ be a subvariety of $Z$. A subvariety $W$ of $V$ is a $\Sigma$-subvariety of $V$ in $Z$ if and only if $\delta(W \times \{x\}) - \delta_{\ws}(W) = \delta(\{x\})$.

Furthermore, if a subvariety $W$ of $V$ is maximal $\Sigma$-atypical, then $W \times \{x\}$ is an optimal subvariety for $V \times \{x\}$ in $Z \times X$. Any maximal $\Sigma$-atypical subvariety of $V$ is weakly optimal for $V$ in $Z$.
\end{prop}

\begin{proof}
Note first that $\langle W \times \{x\} \rangle_{\ws} = \langle W \rangle_{\ws} \times \{x\}$ and $\delta_{\ws}(W \times \{x\}) = \delta_{\ws}(W)$ for every subvariety $W$ of $V$.

It follows from Lemma \ref{lem:sigmaspecialchar} that $W$ is a $\Sigma$-subvariety if and only if $\langle W \times \{x\} \rangle_{\ws}$ is an irreducible component of $S \cap (Z \times \{x\})$ for some special subvariety $S$ of $Z \times X$, which in turn is equivalent to $\langle W \times \{x\} \rangle_{\ws}$ being an irreducible component of $\langle W \times \{x\} \rangle \cap (Z \times\{x\})$. It follows from Lemma 7.4 in \cite{BD2} and the definition of the special closure that every irreducible component of $\langle W \times \{x\} \rangle \cap (Z \times \{x\})$ has dimension $\dim \langle W \times \{x\} \rangle - \dim \langle \{x\} \rangle$.

We deduce that $W$ being a $\Sigma$-subvariety is equivalent to
\[ \dim \langle W \times \{x\} \rangle_{\ws} = \dim \langle W \times \{x\} \rangle - \dim \langle \{x\} \rangle,\]
which in turn is equivalent to
\[ \delta(\{x\}) = \dim \langle \{x\} \rangle = \dim \langle W \times \{x\} \rangle - \dim \langle W \times \{x\} \rangle_{\ws} = \delta(W \times \{x\}) - \delta_{\ws}(W).\]
This establishes the first part of the proposition.

For the second part, let $W$ be a maximal $\Sigma$-atypical subvariety of $V$. There exists a subvariety $U$ of $V$ such that $W \subset U$, $U \times \{x\}$ is optimal for $V \times \{x\}$ in $Z \times X$, and $\delta(U \times \{x\}) \leq \delta(W \times \{x\})$. It follows that $W \times \{x\}$ is optimal for $V \times \{x\}$ in $Z \times X$ if and only if $W = U$.

 We claim that $U$ is a $\Sigma$-subvariety. Theorem 7.2 in \cite{BD2} (the Defect Condition) implies that
\[ \delta(U \times \{x\}) - \delta_{\ws}(U) \leq \delta(W \times \{x\}) - \delta_{\ws}(W) = \delta(\{x\}).\]
On the other hand, $\langle U \times \{x\} \rangle$ projects surjectively onto $\langle \{x\} \rangle$ with the fibers of the restriction  to $\langle U \times\{x\} \rangle$ of the projection map all being of the same dimension equal to at least $\dim \langle U \rangle_{\ws}$. Theorem 7.3 in \cite{BD2} (the Fiber Dimension Theorem) then implies that
\[ \dim \langle U \times\{x\} \rangle -  \dim \langle \{x\} \rangle \geq \dim \langle U \rangle_{\ws}\]
or equivalently
\[ \delta(U \times\{x\}) - \delta_{\ws}(U) \geq \delta(\{x\}).\]
Combining this inequality with the one above, we deduce that equality must hold and so $U$ is a $\Sigma$-subvariety.

It follows that $\delta_{\ws}(U) = \delta(U \times \{x\}) - \delta(\langle\{x\}\rangle) \leq \delta(W \times \{x\}) - \delta(\langle\{x\}\rangle) = \delta_{\ws}(W) < \dim Z - \dim V$ and so $U$ is $\Sigma$-atypical. But since $W$ is maximal $\Sigma$-atypical, we must have that $U = W$ and we are done.

Finally, the last sentence in the second part of the proposition follows from Proposition 8.2 in \cite{BD2} and the fact that $W$ being weakly optimal for $V$ in $Z$ is equivalent to $W \times \{x\}$ being weakly optimal for $V \times \{x\}$ in $Z \times X$.
\end{proof}

\begin{defn}
Let $\Sigma$ be a $\Sigma$-orbit, let $Z$ be a distinguished variety, and let $V$ be a subvariety of $Z$. A subvariety $W$ of $V$ is called \emph{$\Sigma$-optimal (for $V$ in $Z$)} if $W$ is weakly optimal for $V$ in $Z$ as well as a $\Sigma$-subvariety of $V$.
\end{defn}

\begin{rmk}\label{rmk:sigmaatypical}
Let $\Sigma$ be a $\Sigma$-orbit, let $Z$ be a distinguished variety, and let $V$ be a subvariety of $Z$. All maximal $\Sigma$-atypical subvarieties of $V$ are $\Sigma$-optimal for $V$ in $Z$. If the distinguished category satisfies \ref{ax:5}, then Corollary \ref{cor:sigmaspeccont} implies that all maximal $\Sigma$-special subvarieties of $V$ are $\Sigma$-optimal for $V$ in $Z$.
\end{rmk} 

The following theorem is the main result of this section.

\begin{thm}\label{thm:AD1}
	Let $\Sigma$ be a $\Sigma$-orbit in a distinguished category satisfying \ref{ax:5}. Let $Z$ be a distinguished variety and let $\mathcal{S}$ be a class of distinguished varieties such that $Z \in \mathcal{S}$ and for every $X' \in \mathcal{S}$ and every subvariety $V$ of $X'$, the set in \ref{ax:5} can be chosen with all $Y_\phi,Z_\psi$ belonging to $\mathcal{S}$. Fix an integer $m\geq 0$. Assume $\WML(X',\Sigma,m)$ holds for every $X' \in \mathcal{S}$. Then, if $V\subset Z$ is a subvariety of dimension at most $m$, there are at most finitely many $\Sigma$-optimal varieties for $V$ in $Z$.
\end{thm}

The above theorem can be reformulated as follows.

\begin{thm}\label{thm:AD2}		Let $\Sigma$ be a $\Sigma$-orbit in a distinguished category satisfying \ref{ax:5}. Let $Z$ be a distinguished variety and let $\mathcal{S}$ be a class of distinguished varieties such that $Z \in \mathcal{S}$ and for every $X' \in \mathcal{S}$ and every subvariety $V$ of $X'$, the set in \ref{ax:5} can be chosen with all $Y_\phi,Z_\psi$ belonging to $\mathcal{S}$. Fix an integer $m\geq 0$. Assume $\WML(X',\Sigma,m)$ holds for every $X' \in \mathcal{S}$. Let $V$ be a subvariety of $Z$ of dimension at most $m$. There exists a finite collection $\Delta$ of $\Sigma$-special subvarieties of $Z$ such that every proper $\Sigma$-optimal subvariety of $V$ is contained in a subvariety from $\Delta$ and no subvariety from $\Delta$ contains $V$.
\end{thm}

Let us first see that these two statements are equivalent. Theorem \ref{thm:AD1} implies Theorem \ref{thm:AD2} because any proper $\Sigma$-optimal subvariety of $V$ is contained in its weakly special closure, which is a $\Sigma$-special subvariety of $Z$ that does not contain $V$.

The converse is less trivial. We proceed by induction on $m$ and the base of the induction is trivial.

Let $V$ be a subvariety of $Z$ of dimension at most $m$ and let $W$ be a proper $\Sigma$-optimal subvariety of $V$ and $\Delta$ be the set of $\Sigma$-special subvarieties given by Theorem \ref{thm:AD2} for $V$. Let $S$ be a $\Sigma$-special subvariety from $\Delta$ such that $W\subset S$. Let $W'$ be a component of $S\cap V$ containing $W$. We know that $W' \neq V$ since $V \not\subset S$. Now $W$ is also $\Sigma$-optimal for $W'$ in $Z$ and we obtain the claim by induction.

We are now ready to prove Theorem \ref{thm:AD2}.

\begin{proof}[Proof of Theorem \ref{thm:AD2}]
We induct on $m$, the base case being trivial. Let $V$ be a subvariety of $Z$ of dimension $m$.
 For every proper $\Sigma$-optimal subvariety for $V$, its weakly special closure is a $\Sigma$-special subvariety of $Z$ that does not contain $V$. By induction, it therefore suffices to show that there exists some proper closed subset of $V$ that contains every proper $\Sigma$-optimal subvariety of $V$. Let $W$ be one such proper $\Sigma$-optimal subvariety for $V$.

By \ref{ax:5} there are distinguished morphisms $\psi:Y\to Z$ and $\phi:Y\to X$, coming from a finite set, independent of $W$, and $x\in X(K)$ such that $\langle W\rangle_{\ws}$ is a component of $\psi(\phi^{-1}(x))$ and $\psi $ has finite fibers. Note that, by our hypothesis, $\WML(X,\Sigma, m)$ holds.

 We are going to prove that $\psi,\phi$, and $V$ determine a finite set of proper subvarieties of $V$ whose union contains $W$. This will give the claim that is needed to apply the induction hypothesis because $\psi$ and $\phi$ come from a finite set that does not depend on $W$.

Let $U_1, \dots , U_l $ be the components of $\phi^{-1}(x)$. We may assume that $U_1$ dominates $\langle W\rangle_{\ws}$. Since $\psi$ has finite fibers we have $\dim \psi(U_1)=\dim \langle W\rangle_{\ws}$ and thus $U_1$ is a component of $\psi^{-1}(\langle W \rangle_{\ws} )$. We have moreover that $\psi(U_1)= \langle W \rangle_{\ws}$ and that $U_1$ and $\{x\}$ are $\Sigma$-special by Proposition \ref{prop:permanencesigmaspecial}.

We let now $\tilde{W}$ be a component of $\psi|_{U_1}^{-1}(W)$ with $\dim \tilde{W}=\dim W$ and $\tilde{V}$ be a component of $\psi^{-1}(V)$ that contains $\tilde{W}$. Since $U_1$ is weakly special, we have $\delta_{\ws}(\tilde{W})\leq \delta_{\ws}(W)$.

We may assume that $\dim \tilde{V} = \dim V$ since otherwise the closure of $\psi (\tilde{V})$ would be a proper closed subset of $V$ containing $W$ and we would be done by induction.  It follows that $\tilde{W} \neq \tilde{V}$ and that $V$ is contained in the closure of $\psi(\tilde{V})$. This implies that $V\subset \psi ( \langle \tilde{V}\rangle_{\ws}  )$ and thus $\delta_{\ws}(V)\leq \delta_{\ws}(\tilde{V})$.

Let $V'$ denote the closure of $\phi(\tilde{V})$ and recall that $\phi(\tilde{W})=\{x\}$. Suppose first that $\dim \tilde{W} > \dim \tilde{V} - \dim V'$. By the Fiber Dimension Theorem (see Theorem 7.3 of \cite{BD2}), $\tilde{W}$ is contained in a proper closed subset of $\tilde{V}$ independent of $\tilde{W}$ and the same holds for $W$ and $V$. We are then done as before.

From now on, we assume that $\dim \tilde{W} \leq \dim \tilde{V} - \dim V'$. We want to prove by way of contradiction that $V'$ is not weakly special. Assume therefore that $V'$ is weakly special, and let $\hat{V}$ be a component of $\phi^{-1}(V') $ containing $\tilde{V}$. Then, $\hat{V}$ is weakly special (thus $\langle \tilde{W} \rangle_{\ws}\subset \hat{V}$) and we can apply Lemmas 7.5 and 7.6 of \cite{BD2} to $\phi|_{\hat{V}}:\hat{V}\to V'$, $\{x\}$ and $ \langle \tilde{W} \rangle_{\ws}$ (recall that $\langle \tilde{W} \rangle_{\ws}$ is a component of $\phi^{-1}(\{x\})$), and obtain $\dim \hat{V}-\dim V'= \dim \langle \tilde{W} \rangle_{\ws}$.

We use now this equality and what we showed before to deduce that
\begin{multline*}
\delta_{\ws}(V) \leq \delta_{\ws}(\tilde{V}) \leq \dim \hat{V}-\dim \tilde{V}\leq  \dim \hat{V} - \dim V' - \dim \tilde{W}\\
=\dim \langle \tilde{W} \rangle_{\ws}- \dim \tilde{W} =\delta_{\ws}(\tilde{W})\leq \delta_{\ws}(W).
\end{multline*}
This contradicts the fact that $W $ is a proper weakly optimal subvariety for $V$. Thus $V'$ is not weakly special.

We can now finally use $\WML(X,\Sigma, m)$ and conclude the proof.
We have that $\{x\}\subset V'$ and $x$ is a $\Sigma$-special point of $X$. Since $V'$ is not weakly special and $\dim V'\leq \dim \tilde{V}=\dim V=m$, it follows from $\WML(X,\Sigma, m)$ that $\{x\} \subset \bigcup_{i=1}^{n}{T_i}$, where each $T_i$ is a proper subvariety of $V'$ ($i = 1, \hdots, n$). At least one of the $T_i$, say $T_1$, must contain $x$ and therefore
\[
\tilde{W} \subset \phi^{-1}(x) \cap \tilde{V} \subset  \phi^{-1}(T_1) \cap \tilde{V} \subsetneq \tilde{V}
\]
and we are done again because the closure of $ \psi (\phi^{-1}(T_1) \cap \tilde{V} )$, which is independent of $W$, is a proper closed subset of $V$ that contains it.
\end{proof}

\begin{cor}
Let $\Sigma$ be a $\Sigma$-orbit in a distinguished category satisfying \ref{ax:5}. Let $Z$ be a distinguished variety and let $\mathcal{S}$ be a class of distinguished varieties such that $Z \in \mathcal{S}$ and for every $X' \in \mathcal{S}$ and every subvariety $V$ of $X'$, the set in \ref{ax:5} can be chosen with all $Y_\phi,Z_\psi$ belonging to $\mathcal{S}$. Fix an integer $m\geq 0$. Assume $\WML(X',\Sigma,m)$ holds for every $X' \in \mathcal{S}$. Then, $\SML(Z,\Sigma,m)$ holds.
\end{cor}

\begin{proof} Let $V$ be a subvariety of $Z$ of dimension at most $m$ and let $W$ be a maximal $\Sigma$-special subvariety of $V$. By Theorem \ref{thm:AD1} it is enough to prove that $W$ is weakly optimal for $V$ in $Z$. If this were not the case, then $W$ would be strictly contained in a weakly special subvariety of $V$, which must be $\Sigma$-special by Corollary \ref{cor:sigmaspeccont}. This contradicts the fact that $W$ is maximal $\Sigma$-special in $V$.
\end{proof}

The set $\mathcal{S}$ in Theorem \ref{thm:AD1} is a nuisance. We now prove two lemmas that allow us to get rid of it at the cost of assuming that $\SML(Z,\Sigma,m)$ holds for all $m \geq 0$.

\begin{lem}\label{lem:strongax5}
In a distinguished category satisfying \ref{ax:5}, the distinguished morphisms $(\phi,\psi)$ in \ref{ax:5} can be chosen such that each $\psi$ is surjective.
\end{lem}

\begin{proof}
Given $\phi: Y_\phi \to X$ and $\psi: Y_\phi \to Z_\psi$ such that $\phi$ has finite fibers, we apply \ref{ax:4} to find distinguished morphisms $\phi_1: V_\phi \to Y_\phi$, $\psi_1: V_\phi \to U_\psi$, and $\psi_2: U_\psi \to Z_\psi$ such that $\psi \circ \phi_1 = \psi_2 \circ \psi_1$, $\phi_1$ is finite and surjective, $\psi_1$ is surjective, and $\psi_2$ has finite fibers.

Let $W$ be an irreducible component of $\phi(\psi^{-1}(z))$ for some $z \in Z_\psi(K)$. Note that
\[ \phi(\psi^{-1}(z)) = (\phi \circ \phi_1)((\psi \circ \phi_1)^{-1}(z)) = (\phi \circ \phi_1)((\psi_2 \circ \psi_1)^{-1}(z)) = \bigcup_{u \in \psi_2^{-1}(z)}{(\phi \circ \phi_1)(\psi_1^{-1}(u))},\]
where the last union of closed subsets is finite since $\psi_2$ has finite fibers. Hence, $W$ is also an irreducible component of $(\phi \circ \phi_1)(\psi_1^{-1}(u))$ for some $u \in U_\psi(K)$. Therefore, we can replace $(\phi,\psi)$ by $(\phi \circ \phi_1, \psi_1)$, where $\phi \circ \phi_1$ has finite fibers and $\psi_1$ is surjective.
\end{proof}

\begin{lem}\label{lem:smlinherited}
Let $\Sigma$ be a $\Sigma$-orbit and let $X$ be a distinguished variety such that, for all integers $m \geq 0$, $\SML(X,\Sigma,m)$ holds. Let $\phi: Y \to X$ be a distinguished morphism with finite fibers and let $\psi: Y \to Z$ be a surjective distinguished morphism. Then $\SML(Y,\Sigma,m)$ and $\SML(Z,\Sigma,m)$ hold for all integers $m \geq 0$.
\end{lem}

\begin{proof}
We first prove that $\SML(Y,\Sigma,m)$ holds for all integers $m \geq 0$. We proceed by induction on $m$, the case $m = 0$ being trivial. So let $V \subset Y$ be a subvariety of dimension $m$ and let $S \subset V$ be a maximal $\Sigma$-special subvariety. Let $W$ denote the closure of $\phi(V)$. Then, by Proposition \ref{prop:permanencesigmaspecial}, $\phi(S)$ is a $\Sigma$-special subvariety of $W$ and therefore contained in a maximal $\Sigma$-special subvariety $S'$ of $W$, which must come from a finite set since $\SML(X,\Sigma,m)$ holds for all $m$.

 We have $S \subset \phi^{-1}(S')$, so $S$ is contained in some component $S''$ of $\phi^{-1}(S')$. We first assume that $S'' \not\subset V$. Since $\phi$ has finite fibers, we have $\dim S'' \leq \dim S' \leq \dim W = \dim V$. Hence, we cannot have $V \subset S''$ because otherwise $S'' = V$, a contradiction. So $S'' \cap V$ is a proper closed subset of $V$ that contains $S$ and we are again done by induction (note that $S'$ and therefore also $S''$ vary in finite sets respectively).

We can therefore assume that $S'' \subset V$. It follows from the maximality of $S$ that $S = S''$ and so $S$ lies in a finite set and we are done.

We now prove that $\SML(Z,\Sigma,m)$ holds for all integers $m \geq 0$. So let $V \subset Z$ be a subvariety. If $S \subset V$ is a maximal $\Sigma$-special subvariety, then every component of $\psi^{-1}(S)$ is a maximal $\Sigma$-special subvariety of some component of $\psi^{-1}(V)$, which is enough to conclude.
\end{proof}

\section{Applications}\label{sec:applications}

\subsection{Semiabelian varieties} 
The following theorem is one formulation of the Mordell-Lang conjecture for semiabelian varieties.

\begin{thm}[Mordell-Lang for semiabelian varieties, \cite{McQ}, Conjecture 1] \label{thm:MLsemi}
	Fix a semiabelian variety $G$ over $K$ and let $\Gamma$ be a finite rank subgroup of $G(K)$.
	Let $V\subset G$ be any subvariety such that $V(K)\cap \Gamma$ is Zariski dense in $V$. Then $V$ is a coset of a semiabelian subvariety of $G$.
\end{thm}

Given a semiabelian variety $G$, we have seen in Subsection \ref{subsec:semiab1} that $\Sigma_G$, for some $\Sigma $-orbit $\Sigma$, is contained in a finite rank subgroup of $G(K)$ and for any finite rank subgroup $\Gamma\subset G(K)$, there exists a $\Sigma $-orbit $\Sigma$ with $\Gamma \subset \Sigma_G$.
It is therefore quite evident that, for any integer $m\geq 0$, $\WML(G, \Sigma,m)$ for any $\Sigma$-orbit $\Sigma$ is equivalent to the statement of Theorem  \ref{thm:MLsemi} for any finite rank subgroup $\Gamma$ and any subvariety $V$ of dimension at most $m$.

Finally, we want to show that we can recover Theorem 1.8 of \cite{Aslanyan} for semiabelian varieties. 

Let $\Gamma$ be a finite rank subgroup of a semiabelian variety $G$. Recall that we have seen in Subsection \ref{subsec:semiab1} that there exists a $\Sigma$-orbit $\Sigma$ such that $\Gamma\subset \Sigma_G$. A $\Gamma$-special subvariety in the sense of \cite{Aslanyan} is $\Sigma$-special thanks to Lemma \ref{lem:specialimpliespointspecial}.

 Let $V\subset G$ be a subvariety. A maximal $\Gamma$-atypical subvariety of $V$ in the sense of \cite{Aslanyan}  is certainly $\Sigma$-atypical, but it is also maximal $\Sigma$-atypical. Indeed, let $W\subset W' \subset V$ with $W$ maximal $\Gamma$-atypical and $\Sigma$-atypical and $W'$ maximal $\Sigma$-atypical. As $\langle W\rangle_{\ws}\subset\langle W'\rangle_{\ws}$ and $\langle W\rangle_{\ws}$ contains a point of $\Gamma$, $W'$ is $\Gamma$-atypical. This implies that $W=W'$ and we have that $W$ is maximal $\Sigma$-atypical.

It is now clear that, by Proposition \ref{prop:wo} and the considerations of this subsection, our Theorem \ref{thm:AD1} and Theorem \ref{thm:MLsemi} imply Theorem 1.8 of \cite{Aslanyan} for semiabelian varieties.

\subsection{Connected pure Shimura varieties}

In this subsection, we work with the distinguished category of connected pure Shimura varieties.
Let us start by considering the case of products of modular curves and see how to recover Theorem 1.8 of \cite{Aslanyan} for $Y(1)^n_\mathbb{C}$ from our Theorem \ref{thm:AD1}.

Let $Z=Y(1)^n_\mathbb{C}$ for some positive integer $n$ and let $\Sigma = \Sigma(x)$ be a $\Sigma$-orbit for $x \in Y(1)^{n_0}(\mathbb{C})$ for some positive integer $n_0$. Recall that, in this setting, $\Sigma$-special subvarieties of $Z$ are $\Sigma_Z$-point-special by Proposition \ref{prop:pointShimura}. Recall furthermore that we determined $\Sigma_Z$ in Subsection \ref{sub:csv1}. Then, Theorem 6.2 of \cite{Pila_2014} (see also Theorem 3 in \cite{HP12} for an earlier result) implies $\SML(Z, \Sigma, m)$ for all $m\geq 0$. Therefore, also $\WML(Z, \Sigma, m)$ holds for all $m\geq 0$.

Arguing as in the previous subsection and using again Proposition \ref{prop:wo} as well as the considerations of Subsection \ref{sub:csv1}, we see that, given a structure of finite rank $\Gamma$, maximal $\Gamma$-atypical subvarieties of $V\subset Z=Y(1)_\mathbb{C}^n$ in the sense of \cite{Aslanyan} are $\Sigma$-optimal for $V$, where $\Sigma$ is a $\Sigma$-orbit such that $\Gamma \subset \Sigma_Z$.

Now, Theorem \ref{thm:AD1} would imply Theorem 1.8 of \cite{Aslanyan} for $Y(1)_{\mathbb{C}}^n$ if we could find a class $\mathcal{S}$ of distinguished varieties as described in the hypotheses of Theorem \ref{thm:AD1} and such that $\WML(X, \Sigma, m)$ holds for all $X$ in $\mathcal{S}$ and for all $m \geq 0$.

We directly get such a class $\mathcal{S}$ from Lemmas \ref{lem:strongax5} and \ref{lem:smlinherited} in the following way: we start with the set $\{Y(1)_{\mathbb{C}}^n; n \in \mathbb{N}\}$ and iteratively add distinguished varieties that admit a distinguished morphism with finite fibers to a distuingished variety in our set or that admit a surjective distinguished morphism from a distinguished variety from our set.

We now consider a general connected pure Shimura variety $Z$. Suppose that $\Sigma = \Sigma(x)$ for a point $x \in S(\mathbb{C})$, where $S$ is a connected pure Shimura variety. In Theorem 1.4 of \cite{Richard_Yafaev_24b}, Richard and Yafaev have proved that $\WML(Z_{\mathbb{C}},\Sigma,m)$ holds for all connected pure Shimura varieties $Z$ of abelian type and all $m$ (recall that in Subsection \ref{sub:csv1}, we showed the equivalence of the notions of hybrid orbit and $\Sigma$-orbit). This implies that also $\SML(Z_{\mathbb{C}},\Sigma,m)$ holds for all connected pure Shimura varieties $Z$ of abelian type and all $m$ (see Lemma 3.4 in \cite{AD22} and Lemma \ref{lem:wmlequivalentsml}). We can again use Lemmas \ref{lem:strongax5} and \ref{lem:smlinherited} to get a class $\mathcal{S}$ that contains all connected pure Shimura varieties of abelian type and satisfies the condition in Theorem \ref{thm:AD1}. Therefore, we obtain the following corollary of Theorem \ref{thm:AD1} together with the work of Richard and Yafaev.

\begin{cor}\label{cor:shimura}
	Suppose that $\Sigma = \Sigma(x)$ for a point $x \in S(\mathbb{C})$, where $S$ is a connected pure Shimura variety. Let $Z$ be a connected pure Shimura variety of abelian type and let $V\subset Z$ be a subvariety. Then there are at most finitely many $\Sigma$-optimal varieties for $V$ in $Z$.
\end{cor}

Recall from Subsection \ref{sub:csv1} that for every structure of finite rank $\widetilde{\Sigma} \subset Z(\mathbb{C})$ as defined in \cite{AD22}, there is a $\Sigma$-orbit $\Sigma = \Sigma(x)$ such that $\widetilde{\Sigma} \subset \Sigma_{Z_{\mathbb{C}}}$ and actually one can take $x$ to lie in $Z^n(\mathbb{C})$ for some $n \in \mathbb{N}$. Because of Proposition \ref{prop:pointShimura} and Remark \ref{rmk:sigmaatypical}, we then recover Theorem 1.2 of \cite{AD22}.

\subsection{Powers of elliptic schemes}

Let $\pi: \mathcal{E} \to Y(2) = \mathbb{A}^1_{\bar{\mathbb{Q}}} \backslash\{0,1\}$ denote the Legendre family of elliptic curves over $\bar{\mathbb{Q}}$, defined in $\mathbb{P}^2_{\bar{\mathbb{Q}}} \times Y(2)$ by the equation $Y^2Z = X(X-Z)(X-\lambda Z)$, where $[X:Y:Z]$ are the projective coordinates on $\mathbb{P}^2_{\bar{\mathbb{Q}}}$ and $\lambda$ is the affine coordinate on $Y(2)$. For $n \in \mathbb{N}$, we denote by $\mathcal{E}^{(n)}$ the $n$-th fibered power of $\mathcal{E}$ over $Y(2)$ and by $\mathcal{E}^n$ the $n$-th direct power of $\mathcal{E}$ over $\bar{\mathbb{Q}}$. By abuse of notation, we also use $\pi$ to denote the projections $\mathcal{E}^{(n)} \to Y(2)$ and $\mathcal{E}^n \to Y(2)^n$. All the $\mathcal{E}^{(n)}$, $\mathcal{E}^n$, and $Y(2)^n$ are connected mixed Shimura varieties of Kuga type over $\bar{\mathbb{Q}}$ (see Section 13 of \cite{BD2}).

We fix a point $x = (x_1,\hdots,x_r)\in \mathcal{E}^r(K)$ for some algebraically closed field $K \supset \bar{\mathbb{Q}}$ and some $r \in \mathbb{N}$ and let $\Sigma=\Sigma(x)$ in the distinguished category of connected mixed Shimura varieties of Kuga type over $K$. We also let $E$ denote an arbitrary elliptic curve with complex multiplication over $\bar{\mathbb{Q}}$. Note that $E^n$ is a connected mixed Shimura variety of Kuga type for all $n \in \mathbb{N}$.

We claim that, for $S$ equal to $\mathcal{E}^{(n)}_K$, $\mathcal{E}^n_K$, $Y(2)^n_K$, and $E_K^n$ respectively, the $\Sigma$-special subvarieties of $S$ can be explicitly described like this: first, a subvariety $Y$ of $Y(2)_K^n$ is $\Sigma$-special if and only if there exist $s,t \in \mathbb{Z}_{\geq 0}$ such that, after a permutation of the coordinates, $Y$ becomes equal to $Y_1 \times \cdots \times Y_s \times \{(y_1,\hdots,y_t)\}$, where for every $i = 1, \hdots, s$, $Y_i$ is a $1$-dimensional irreducible component of the preimage in $Y(2)_K^{m_i}$ under the canonical Shimura morphism $Y(2)_K^{m_i} \to Y(1)_K^{m_i}$ of the common zero locus in $Y(1)_K^{m_i}$ of a finite set of modular polynomials evaluated at pairs of distinct coordinates, $\sum_{i=1}^{s}{m_i} + t = n$, and for every $i = 1,\hdots, t$, either $y_i \in Y(2)(K)$ is special or $\mathcal{E}_{y_i}$ is isogenous to $\mathcal{E}_{\pi(x_j)}$ for some $j \in \{1,\hdots,r\}$. Second, a subvariety $Z$ of $\mathcal{E}_K^n$ is $\Sigma$-special if and only if $\pi(Z)$ is a $\Sigma$-special subvariety of $Y(2)_K^n$ and, after a permutation of the coordinates, $Z$ is equal to $Z_0 \times Z_1$ where $Z_0 \subseteq \mathcal{E}_K^{n-s}$ and $Z_1 \subseteq \mathcal{E}^s_K$ ($0 \leq s \leq n$) and the geometric generic fiber of $Z_0 \to \pi(Z_0)$ is a union of irreducible components of algebraic subgroups while $\pi(Z_1) = \{(y_1, \hdots, y_s)\}$ for some $y_1, \hdots, y_s \in Y(2)(K)$ and $Z_1$ is a translate of an abelian subvariety of $\mathcal{E}_{y_1} \times \cdots \times \mathcal{E}_{y_s}$ by a point $(z_1, \hdots, z_s)$ such that for all $i = 1,\hdots, s$, there exist $N_i \in \mathbb{N}$ and homomorphisms of algebraic groups $\phi_{i,j}: \mathcal{E}_{\pi(x_j)} \to \mathcal{E}_{y_i}$ ($j = 1,\hdots, r$) such that $N_iz_i = \sum_{j=1}^{r}{\phi_{i,j}(x_j)}$ (note that $\phi_{i,j} = 0$ is allowed). Finally, a subvariety of $\mathcal{E}_K^{(n)}$ or of $E_K^n$ is $\Sigma$-special if and only if it is $\Sigma$-special when considered as a subvariety of $\mathcal{E}_K^n$ (where we identify $E$ with some fiber of $\mathcal{E} \to Y(2)$ that is isomorphic to it -- the choice of the fiber does not matter).

Indeed, this claim follows from Lemma \ref{lem:sigmaspecialchar} together with the characterization of special subvarieties of products of modular curves in Propositions 2.1 and 3.1 in \cite{Edixhoven_2005} and the characterization of Shimura subdata of connected mixed Shimura data of Kuga type in Proposition 1.2.16 in \cite{GDiss}.

\begin{thm}\label{thm:MLlegendre}
	Recall that $\Sigma = \Sigma(x)$ for a fixed point $x = (x_1,\hdots,x_r) \in \mathcal{E}^r(K)$. Let $n \in \mathbb{N}$ and let $m$ be a non-negative integer. Then $\WML(\mathcal{E}^{(n)}_K,\Sigma,m)$ holds if one of the following conditions is satisfied:
	\begin{enumerate}
		\item $x_1,\hdots,x_r$ are all special points, or
		\item $K = \bar{\mathbb{Q}}$ and for every $i = 1,\hdots,r$, if $\mathcal{E}_{\pi(x_i)}$ has complex multiplication, then $x_i$ is a special point in $\mathcal{E}(\bar{\mathbb{Q}})$, or
		\item $m \leq 1$.
	\end{enumerate}
\end{thm}

\begin{proof}
	We apply our above characterization of the $\Sigma$-special subvarieties of $\mathcal{E}^{(n)}_K$. We also use again Proposition 5.3 in \cite{G182}, which yields a characterization of the weakly special subvarieties of $\mathcal{E}^{(n)}_K$. Under condition (1), $\Sigma$ is just the set of special points of $\mathcal{E}^{(n)}_K$ and Mordell-Lang reduces to Andr\'e-Oort. In this case the claim follows from Theorem 13.6 in \cite{G17}. Under condition (2), we deduce the theorem by combining Theorem 13.6 in \cite{G17} with Theorem 1.1 in \cite{Dill20}.
	Indeed, given a subvariety $V$ of $ \mathcal{E}^{(n)}_K$, if it contains a dense set of special points, then it is special by Andr\'e-Oort. Otherwise, in $V$ there must be a dense set of non-special points $y$ that are $\Sigma$-special. This means that, for any such $y=(y_1 ,\dots , y_n)$, $\mathcal{E}_{\pi(y)}$ is isogenous to one of the $\mathcal{E}_{\pi(x_k)}$ without CM and relations of the form $N_iy_i = \sum_{j=1}^{r}{\phi_{i,j}(x_j)}$ hold, where the $N_i$ are positive integers and $\phi_{i,j}: \mathcal{E}_{\pi(x_j)} \to \mathcal{E}_{\pi(y_i)}$ ($j = 1,\hdots, r$) are homomorphisms of algebraic groups ($i = 1,\hdots,n$). Our claim then follows from Theorem 1.1 in \cite{Dill20}.

	 Under condition (3), if $K = \bar{\mathbb{Q}}$, the theorem follows from Theorem 13.6 in \cite{G17} and Theorem 1.3 in \cite{Dill21}. For general $K$, a sketch of the proof of an analogue of Theorem 1.3 in \cite{Dill21} can be found in Section 5.1 in \cite{DillThesis}.
\end{proof}

\begin{rmk}
See also Theorem 1.6 in \cite{G15} for a related result and see the master's thesis \cite{NB} of N\'estor Bl\'azquez for a detailed characterization of the weakly special subvarieties of $\mathcal{E}^{(n_1)}_{\mathbb{C}} \times \mathcal{E}^{(n_2)}_{\mathbb{C}}$ as well as a proof of Theorem \ref{thm:MLlegendre} under condition (2) without the assumption that $K = \bar{\mathbb{Q}}$.
\end{rmk}

\begin{cor}\label{cor:maincor}
Let $n \in \mathbb{N}$. Suppose that $K$ and $x$ satisfy condition (1) or (2) in Theorem \ref{thm:MLlegendre}. Then, for every subvariety $V\subset \mathcal{E}^{(n)}_K$, there are at most finitely many $\Sigma$-optimal varieties for $V$ in $\mathcal{E}^{(n)}_K$.
\end{cor}

Corollary \ref{cor:maincor} also holds if $V$ is a curve so that condition (3) from Theorem \ref{thm:MLlegendre} is satisfied, but in that case, the conclusion of Corollary \ref{cor:maincor} is equivalent to $\WML(\mathcal{E}^{(n)}_K,\Sigma,1)$. 

\begin{proof}
We want to combine Theorems \ref{thm:AD1} and \ref{thm:MLlegendre}. We have to verify that $\WML(X',\Sigma,m)$ holds for all $m \geq 0$ and every distinguished variety $X'$ in a class $\mathcal{S}$ that verifies the hypothesis in Theorem \ref{thm:AD1} for $Z = \mathcal{E}^{(n)}_K$. By Lemmas \ref{lem:strongax5} and \ref{lem:smlinherited}, it suffices to verify that $\SML(\mathcal{E}^{(n)}_K,\Sigma,m)$ holds for all $m \geq 0$ and $n \in \mathbb{N}$. Since every $\Sigma$-special subvariety of $\mathcal{E}^{(n)}_K$ contains a Zariski dense set of $\Sigma$-special points by the classification above, this follows from Lemma \ref{lem:wmlequivalentsml} and Theorem \ref{thm:MLlegendre}.
\end{proof}

\section*{Acknowledgements}
We thank Lars K\"uhne for useful discussions. The second-named author thanks Salim Tayou for a discussion in Cetraro that led to the appendix of this article.

The first-named author was supported by the  PRIN 2022 project “2022HPSNCR: Semiabelian varieties, Galois representations and related Diophantine problems”
and the ``National Group for Algebraic and Geometric Structures, and their Applications" (GNSAGA INdAM).

This material is based upon work supported by the National Science Foundation under Grant No. DMS--1928930 while the second-named author was in residence at the Mathematical Sciences Research Institute in Berkeley, California, during the Spring 2023 semester.

The second-named author thanks the DFG for its support (grant no. EXC-2047/1 - 390685813).

\noindent
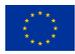
\framebox[\textwidth]{
	\begin{tabular*}{0.96\textwidth}{@{\extracolsep{\fill} }cp{0.84\textwidth}}
		% The EU emblem
		\raisebox{-0.7\height}{%
			\begin{tikzpicture}[y=0.80pt, x=0.8pt, yscale=-1, inner sep=0pt, outer sep=0pt, 
				scale=0.12]
				\definecolor{c003399}{RGB}{0,51,153}
				\definecolor{cffcc00}{RGB}{255,204,0}
				\begin{scope}[shift={(0,-872.36218)}]
					\path[shift={(0,872.36218)},fill=c003399,nonzero rule] (0.0000,0.0000) rectangle (270.0000,180.0000);
					\foreach \myshift in 
					{(0,812.36218), (0,932.36218), 
						(60.0,872.36218), (-60.0,872.36218), 
						(30.0,820.36218), (-30.0,820.36218),
						(30.0,924.36218), (-30.0,924.36218),
						(-52.0,842.36218), (52.0,842.36218), 
						(52.0,902.36218), (-52.0,902.36218)}
					\path[shift=\myshift,fill=cffcc00,nonzero rule] (135.0000,80.0000) -- (137.2453,86.9096) -- (144.5106,86.9098) -- (138.6330,91.1804) -- (140.8778,98.0902) -- (135.0000,93.8200) -- (129.1222,98.0902) -- (131.3670,91.1804) -- (125.4894,86.9098) -- (132.7547,86.9096) -- cycle;
				\end{scope}
			\end{tikzpicture}
		}
		&
		Gabriel Dill has received funding from the European Research Council (ERC) under the European Union's Horizon 2020 research and innovation programme (grant agreement n$^\circ$ 945714).
	\end{tabular*}
}

\appendix

\section{Unlikely intersections with non-$\bar{\mathbb{Q}}$ subvarieties of $\mathcal{A}_g$}\label{appendix}

In this appendix, we first prove a general statement about certain distinguished categories and we then show how it yields Theorem \ref{thm:transcendentalzilberpinka_g}.

\begin{thm}\label{thm:transcendentalzilberpink}
Let $X$ be a distinguished variety in a distinguished category $\mathfrak{C}$ over $K$ such that the base change $\mathfrak{C}_{K'}$ satisfies \ref{ax:5} whenever $K'/K$ is an extension of algebraically closed fields of finite transcendence degree. Suppose that for all distinguished morphisms $\phi: Y \to X$, $\psi: Y \to Z$ such that $\phi$ has finite fibers and $\phi(Y) = X$, then either $\dim \psi(Y) = 0$ or $\psi$ has finite fibers.

Let $L/K$ be an extension of algebraically closed fields and let $V \subset X_L$ be a subvariety that is not the base change of a subvariety of $X$ (in particular, $K \neq L$). Suppose that $V$ is not contained in a proper special subvariety of $X_L$. Then the union of the intersections of $V$ with all special subvarieties of $X_L$ of codimension $> \dim V$ is not Zariski dense in $V$.
\end{thm}

\begin{proof}
We can assume without loss of generality that $L/K$ is an extension of finite transcendence degree. By induction on the transcendence degree, we can then assume furthermore that $L/K$ is an extension of transcendence degree $1$.

By Lemma 2.2 in \cite{BD}, there is then a subvariety $V_0 \subset X$ such that $V \subset (V_0)_L$ and $\dim V_0 = \dim V + 1$.

Suppose that $w \in V(L)$ belongs to a special subvariety of $X_L$ of codimension $> \dim V$. If $w $ is the base change of a point of $ X(K)$, then $w$ belongs to a proper closed subset of $V$, so we can assume without loss of generality that this is not the case. By Lemma 2.2 in \cite{BD}, there exists a subvariety $W_0 \subset V_0$ such that $w \in (W_0)_L(L)$ and $\dim W_0 = 1$. Furthermore, $\langle (W_0)_L \rangle = \langle \{w\} \rangle$. It follows that 
\begin{equation}\label{eq:defectineq}
    \delta(W_0) < \dim X - \dim V-1 = \delta(V_0).
\end{equation}
Let $U_0 \subset V_0$ be a subvariety such that $W_0 \subset U_0$, 
$\delta(U_0) \leq \delta(W_0)$, and $U_0$ is optimal for $V_0$ in $X$. Then $U_0 \neq V_0$ because of $\delta(U_0) \leq \delta(W_0)$ and inequality \eqref{eq:defectineq} and $U_0$ is also weakly optimal for $V_0$ in $X$ by Proposition 8.2 in \cite{BD2}.

Since the distinguished category we are working in satisfies \ref{ax:5} by hypothesis, we have that $\langle U_0 \rangle_{\mathrm{ws}}$ is an irreducible component of $\phi(\psi^{-1}(z))$ for one of finitely many pairs $(\phi,\psi)$ of distinguished morphisms $\phi: Y_\phi \to X$ and $\psi: Y_\phi \to Z_\psi$ such that $\phi$ has finite fibers and some $z \in Z_\psi(K)$. The set of such pairs only depends on $V_0$.

If $\phi(Y_\phi) \neq X$, then $w \in \phi(Y_\phi) \cap V$, which is a proper closed subset of $V$, and we are done. We can then assume that $\phi(Y_\phi) = X$. But in that case, our hypothesis implies that either $\dim \psi(Y) = 0$ or $\psi$ has finite fibers. In the first case, we must have $\psi(Y_\phi) = \{z\}$ and so $\langle U_0 \rangle_{\mathrm{ws}} = \phi(Y_\phi) = X$, which contradicts the fact that
\[ \delta(U_0) \leq \delta(W_0) < \dim X - \dim V -1 \leq \dim X - \dim U_0.\]
In the second case, it follows that $\psi^{-1}(z)$ is a finite set of points and so $\langle U_0 \rangle_{\mathrm{ws}}$ must be a point, which contradicts the fact that $\dim U_0 \geq \dim W_0 = 1$. We are done.
\end{proof}

We now show that Theorem \ref{thm:transcendentalzilberpink} can be applied to connected pure Shimura varieties with simple adjoint group.

\begin{thm}\label{thm:transcendentalzilberpinkshimura}
Let $S$ be a connected pure Shimura variety whose associated $\mathbb{Q}$-algebraic group has (non-trivial) simple adjoint group. Then, for all distinguished morphisms $\phi: Y \to S$, $\psi: Y \to Z$ such that $\phi$ has finite fibers and $\phi(Y) = S$, then either $\dim \psi(Y) = 0$ or $\psi$ has finite fibers.

In particular, if $K$ is an algebraically closed field of characteristic $0$ and $V \subset S_K$ is a subvariety that is not the base change of a subvariety of $S$ (in particular, $K \neq \bar{\mathbb{Q}}$) and is not contained in a proper special subvariety of $S_K$, we have that the union of the intersections of $V$ with all special subvarieties of $S_K$ of codimension $> \dim V$ is not Zariski dense in $V$.
\end{thm}

\begin{proof} Because of Theorem 9.6 in \cite{BD2}, the second claim follows directly from the first one and from Theorem \ref{thm:transcendentalzilberpink}.

Let $(G,X^+)$ be a connected pure Shimura datum such that $S = \mathcal{F}_{\mathrm{pSv}}(G,X^+,\Gamma)$ for some congruence subgroup $\Gamma \subset G(\mathbb{Q})$.

So suppose that $\phi: (H_1,Y_1^+,\Delta_1) \to (G,X^+,\Gamma)$ and $\psi: (H_1,Y_1^+,\Delta_1) \to (H_2,Y_2^+,\Delta_2)$ are morphisms in $\mathfrak{M}_{\mathrm{pSv}}$ such that $\mathcal{F}_{\mathrm{pSv}}(\phi)$ is surjective and has finite fibers. We have to show that $\mathcal{F}_{\mathrm{pSv}}(\psi)$ has either finite fibers or zero-dimensional image. In the following, by abuse of notation, we will also use $\phi$ and $\psi$ to denote the corresponding Shimura morphisms $(H_1,Y_1^+) \to (G,X^+)$ and $(H_1,Y_1^+) \to (H_2,Y_2^+)$ respectively.

Let $(G^{\mathrm{ad}},X^{+,\mathrm{ad}})$ denote the adjoint connected pure Shimura datum of $(G,X^+)$.

We want to show that the induced homomorphism $H_1 \to G^{\mathrm{ad}}$ is surjective. Assume, by way of contradiction, that this is not the case. We proceed with a similar argument as in the proof of Lemma 4.5 in \cite{BD2}. Since $\mathcal{F}_{\mathrm{pSv}}(\phi)$ is surjective and $(G,X^+) \to (G^{\mathrm{ad}},X^{+,\mathrm{ad}})$ is a Shimura quotient morphism, we can cover $X^{+,\mathrm{ad}}$ with countably many $G^{\mathrm{ad}}(\mathbb{Q})$-translates of the image of $Y_1^{+}$ in $X^{+,\mathrm{ad}}$. It follows that there is some algebraic subgroup $H \subsetneq G^{\mathrm{ad}}$ such that each $x \in X^{+,\mathrm{ad}}$ factors through $\gamma H_{\mathbb{C}} \gamma^{-1}$ for some $\gamma \in G^{\mathrm{ad}}(\mathbb{Q})$. Fix some $x_0 \in X^{+,\mathrm{ad}}$, then, for each $\gamma \in G^{\mathrm{ad}}(\mathbb{Q})$, the set of $g \in G^{\mathrm{ad}}(\mathbb{R})^+$ such that $gx_0$ factors through $\gamma H_{\mathbb{C}} \gamma^{-1}$ is a closed real analytic subset of the real analytic manifold $G^{\mathrm{ad}}(\mathbb{R})^+$ (see the proof of Lemma 4.5 in \cite{BD2}). Since the union of all these subsets for varying $\gamma$ is equal to $G^{\mathrm{ad}}(\mathbb{R})^+$, it follows from the Baire Category Theorem that one of them has non-empty interior. But then, by the identity theorem for real analytic functions, this subset has to be equal to all of $G^{\mathrm{ad}}(\mathbb{R})^+$.

Thus, after replacing $H$ by $\gamma H \gamma^{-1}$ for some $\gamma \in G^{\mathrm{ad}}(\mathbb{Q})$, we can assume that there is some algebraic subgroup $H \subsetneq G^{\mathrm{ad}}$ such that every $x \in X^{+,\mathrm{ad}}$ factors through $H_{\mathbb{C}}$. At the same time, $X^{+,\mathrm{ad}}$ is invariant under the action of $G^{\mathrm{ad}}(\mathbb{R})^+$ by conjugation and therefore every $x \in X^{+,\mathrm{ad}}$ factors through
\[ \bigcap_{g \in G^{\mathrm{ad}}(\mathbb{R})^+}{gH_{\mathbb{C}}g^{-1}}.\]
But $G^{\mathrm{ad}}(\mathbb{R})^+$ is Zariski dense in $G^{\mathrm{ad}}_{\mathbb{R}}$ because of Corollary 5.3 in \cite{MilneISV} and Theorem 17.93 in \cite{MilneAG} and so
\[ \bigcap_{g \in G^{\mathrm{ad}}(\mathbb{R})^+}{gH_{\mathbb{C}}g^{-1}} = \bigcap_{g \in G^{\mathrm{ad}}(\mathbb{C})}{gH_{\mathbb{C}}g^{-1}}.\]
The latter is a closed subset of ${G}^{\mathrm{ad}}_{\mathbb{C}}$ that is $\mathrm{Aut}(\mathbb{C}/\mathbb{Q})$-invariant and therefore equals the base change of a normal algebraic subgroup of $G^{\mathrm{ad}}$. So we can assume without loss of generality that $H$ is normal in  $G^{\mathrm{ad}}$ and, after replacing $H$ by its neutral component, that $H$ is connected. Since $G^{\mathrm{ad}}$ is simple by hypothesis, this implies that $H$ is trivial. But every $x \in X^{+,\mathrm{ad}}$ factors through $H_{\mathbb{C}}$ and so $X^{+,\mathrm{ad}}$ is a singleton, consisting of the zero homomorphism. However, because of (MSD.d), this means that $G^{\mathrm{ad}}(\mathbb{R})$ is compact, a contradiction with (MSD.d) and the hypothesis that $G^{\mathrm{ad}}$ is non-trivial. We deduce that the induced homomorphism $H_1 \to G^{\mathrm{ad}}$ must be surjective and therefore factors through $H_1^{\mathrm{ad}}$.

Let $(H_1^{\mathrm{ad}},Y_1^{+,\mathrm{ad}})$ denote the adjoint connected pure Shimura datum of $(H_1,Y_1^+)$. By Theorem 21.51 in \cite{MilneAG}, the semisimple group $H_1^{\mathrm{ad}}$ is isomorphic to an almost-direct product of almost-simple algebraic groups. Since $H_1^{\mathrm{ad}}$ is adjoint and an almost-simple algebraic group with trivial center is simple (the action by conjugation on any finite normal subgroup has infinite kernel and is therefore necessarily trivial), this is actually a direct product of simple factors. It also follows from Theorem 21.51 in \cite{MilneAG} and the fact that $G^{\mathrm{ad}}$ is simple that the homomorphism $H_1^{\mathrm{ad}} \to G^{\mathrm{ad}}$ must be the projection to one of the simple factors of $H_1^{\mathrm{ad}}$ followed by an isomorphism between that factor and $G^{\mathrm{ad}}$. Finally, since $\mathcal{F}_{\mathrm{pSv}}(\phi)$ has finite fibers, there can be no other simple factors in the product decomposition of $H_1^{\mathrm{ad}}$: if $(H,Y^+)$ were such a factor (we obtain $Y^+$ by taking the quotient of $(H_1^{\mathrm{ad}},Y_1^{+,\mathrm{ad}})$ by the product of all the other factors), then, by (MSD.d), the kernel of the action of $H(\mathbb{R})^+$ on $Y^+$ would be compact, but, again by (MSD.d), $H(\mathbb{R})$ is not compact, so $Y^+$ would be positive-dimensional and $\mathcal{F}_{\mathrm{pSv}}(\phi)$ could not have finite fibers. Thus, the homomorphism $H_1^{\mathrm{ad}} \to G^{\mathrm{ad}}$ must already be an isomorphism.

The homomorphism $H_1 \to H_2$ factors as the composition of a surjective homomorphism $H_1 \to H_3$ and an injective homomorphism $H_3 \to H_2$. The surjective homomorphism $H_1 \to H_3$ induces a surjective homomorphism $H_1^{\mathrm{ad}} \to H_3^{\mathrm{ad}}$. Since $H_1^{\mathrm{ad}} \simeq G^{\mathrm{ad}}$ is simple, the kernel of $H_1^{\mathrm{ad}} \to H_3^{\mathrm{ad}}$ is either trivial or equal to $H_1^{\mathrm{ad}}$. In the first case, the kernel of $H_1 \to H_3$ is contained in the center of $H_1$ and so $\psi$ is a Shimura immersion and $\mathcal{F}_{\mathrm{pSv}}(\psi)$ has finite fibers by Proposition 3.6(2) in \cite{BD2} while in the second case, the image of $H_1 \to H_3$ is contained in the center of $H_3$ and so $\mathcal{F}_{\mathrm{pSv}}(\psi)$ has zero-dimensional image.
\end{proof}

\begin{rmk}\label{rmk:transcendentalzilberpinka_g}
In particular, Theorem \ref{thm:transcendentalzilberpinkshimura} applies to the coarse moduli space of principally polarized abelian varieties of dimension $g \in \mathbb{N}$ since the adjoint group of $\mathrm{GSp}_{2g,\mathbb{Q}}$ is simple (see Summary 21.98 in \cite{MilneAG}).
\end{rmk}

\bibliographystyle{amsalpha}

\bibliography{Bibliography}

\end{document}